 \documentclass[11pt]{amsart}
\usepackage{geometry}   
\usepackage{setspace}
\usepackage{graphicx}
\usepackage{amssymb, amsthm}
\usepackage{amsmath}
\usepackage{hyperref}
\usepackage{epstopdf}
\usepackage{pinlabel}
\usepackage{subfig}
\usepackage[usenames,dvipsnames,svgnames]{xcolor}
\usepackage{tikz}

\usetikzlibrary{arrows,snakes,shapes.geometric,backgrounds} 
\usetikzlibrary[arrows,snakes,shapes.geometric] 

 \newtheorem{theorem}{Theorem}[section]
 \newtheorem*{maintheorem}{Main Theorem}
    \newtheorem{corollary}[theorem]{Corollary}
   \newtheorem{lemma}[theorem]{Lemma}
    \newtheorem{proposition}[theorem]{Proposition}

    \theoremstyle{definition}
\newtheorem{definition}[theorem]{Definition}

\newtheorem{example}[theorem]{Example}

\newtheoremstyle{dotless}{}{}{}{}{\bfseries}{}{ }{}
\theoremstyle{dotless}
\newtheorem*{reptheorem}{Theorem}

  \newcommand{\Z}{\ensuremath{{\mathbb{Z}}}}
\newcommand{\R}{\ensuremath{{\mathbb{R}}}}
\newcommand{\N}{\ensuremath{{\mathbb{N}}}}

  \newcommand{\G}{\Gamma}
  
\newcommand{\AG}{A_\G}            

\newcommand{\ra}{\rightarrow}
\newcommand{\co}{\colon\thinspace}

\title{Contracting boundaries of CAT(0) spaces}
\author{Ruth Charney and Harold Sultan}

\thanks {R. Charney was partially supported by NSF grant DMS-1106726}

\begin{document}

\begin{abstract}  As demonstrated by Croke and Kleiner, the visual boundary of a CAT(0) group is not well-defined since quasi-isometric CAT(0) spaces can have non-homeomorphic boundaries.    We introduce a new type of boundary for a CAT(0) space, called the contracting boundary, made up rays satisfying one of five hyperbolic-like properties. We prove that these properties are all equivalent and that the contracting boundary is a quasi-isometry invariant.  We use this invariant to distinguish the quasi-isometry classes of certain right-angled Coxeter groups.  
\end{abstract}
\maketitle


\section{Introduction}

Boundaries of hyperbolic spaces play a central role in the study of hyperbolic groups. The visual boundary of a hyperbolic metric space consists of equivalence classes of geodesic rays, where two rays are equivalent if they stay bounded distance from each other.  As noted by Gromov in \cite{gromov},  quasi-isometries of hyperbolic metric spaces induce homeomorphisms on their boundaries, thus giving rise to a well-defined notion of the boundary of a hyperbolic group. 
(See \cite{bridsonhaefliger} for a complete proof). 


The visual boundary of a CAT(0) space can be defined similarly.  However, as shown by the striking example of Croke and Kleiner \cite{crokekleiner}, in the CAT(0) setting, boundaries are not quasi-isometry invariant and hence one cannot talk about the boundary of a CAT(0) group.   In this paper we introduce the notion of a \emph{contracting boundary} for a CAT(0) space.  This boundary encodes information about geodesics in the CAT(0) space that behave similarly to hyperbolic geodesics.  Indeed, if the space happens to be hyperbolic, then the contracting boundary is equal to the visual boundary.   

The goal of the paper is to show that the contracting boundary enjoys many of the properties satisfied by boundaries of hyperbolic spaces.  In particular, a quasi-isometry of CAT(0) spaces induces a homeomorphism on their contracting boundaries and hence, the contracting boundary of a CAT(0) group is well-defined.  If the group contains a rank one isometry, its contracting boundary is non-empty and gives an effective, new quasi-isometry invariant.  We demonstrate this with some examples of right-angled Coxeter groups whose quasi-isometry classes can be distinguished using this invariant.

A geodesic $\alpha$ in a CAT(0) space  $X$ is \emph{contracting} if there exists a constant $D$ such that for any metric ball $B$ not intersecting $\alpha$, the projection of $B$ on $\alpha$ has diameter at most $D$.  Set theoretically, the contracting boundary, $\partial_c X$,  consists of points on the visual boundary of $X$ represented by contracting rays.  The set of rays at a basepoint that are $D$-contracting for a fixed $D$ defines a closed subspace $\partial^D_cX \subset \partial X$.  We endow the contracting boundary with the direct limit topology from these subspaces.  While  $\partial_cX$ is not, in general, compact, If $X$ is proper, then it is $\sigma$-compact, that is, it is the union of countably  many compact subspaces.  

The contracting boundary also satisfies a strong visibility property.  Namely, given a
contracting ray $\alpha$ and an arbitrary ray $\beta$, there exists a geodesics $\gamma$ such that $\gamma$ is asymptotic to $\alpha$ in one direction and asymptotic to $\beta$ in the other.  In summary, we prove

\begin{maintheorem} Given a proper, CAT(0) space X, the contracting boundary $\partial_{c}X$, equipped with the direct limit topology, is 
\begin{enumerate}
\item $\sigma$-compact,
\item a visibility space, and 
\item a quasi-isometry invariant.
\end{enumerate} 
\end{maintheorem}

One ingredient of this paper that may be of independent interest is a proof of the equivalence of various hyperbolic type properties for geodesics.  Throughout the literature a robust approach for studying spaces of interest is to first identify a class of geodesics that share features in common with geodesics in hyperbolic spaces.  There are various well-studied properties  that can be used to define precise notions of ``hyperbolic type" geodesics  including the Morse property, the contracting property, superlinear divergence, and slimness (see Section  \ref{sec:hyperbolictype} for definitions). These notions have proved fruitful in analyzing right angled Artin groups \cite{behrstockcharney}, Teichm\"uller space \cite{behrstock,brockfarb,brockmasur,bmm1}, the mapping class group \cite{behrstock}, CAT(0) spaces \cite{behrstockdrutu,bestvinafujiwara,sultanmorse}, and Out($F_{n}$) \cite{algomkfir} among others (see also \cite{drutumozessapir,drutusapir,kapovitchleeb,mm1}).  

In this paper,  we introduce a variation on divergence, called \emph{lower divergence} which captures more subtle behavior of the geodesic and makes sense for rays, as well as geodesic lines.  We prove the following theorem, which extends various prior results.

 \begin{reptheorem}{\bf\ref{theorem:equivalenceexpanded}.}
\emph{Let X be a CAT(0) space and $\gamma \subset X$ a geodesic ray or line.  Then the following are equivalent:
\begin{enumerate}
\item $\gamma$ is $D$--contracting, 
\item $\gamma$ is $M$--Morse,
\item $\gamma$ is $S$--slim,
\item $\gamma$ has superlinear lower divergence, and
\item $\gamma$ has at least quadratic lower divergence.
\end{enumerate}
Moreover,  the constants $D$ and $M$ in parts (1) and (2) determine each other. }
\end{reptheorem}

We remark that the last statement of the theorem (proved in Theorem \ref{theorem:equivalence}) is  crucial in proving continuity of the map on contracting boundaries induced by a quasi-isometry.
   The fact that (1) implies (2) is a well known result, an explicit proof of which is given by Algom-Kfir in  \cite{algomkfir}.  In \cite{sultanmorse}, the second author shows that (2) implies (1), but without explicit control on the constants.   In \cite{bestvinafujiwara}, Bestvina and Fujiwara develop many properties of contracting geodesics and, in particular, prove that  (1) implies (3).   Related theorems also appear in  \cite{behrstock,drutumozessapir,kapovitchleeb}, though the context varies among these papers.

In the last section of the paper, we consider the case of a CAT(0) cube complex.  Recent groundbreaking work of Wise,  Agol, Groves, Manning and others has focused much attention on these spaces and shown that a wide range of groups act on such complexes.  In the case of a CAT(0) cube complex with a bounded number of cells at each vertex, we give an explicit combinatorial criterion for determining when a geodesic is contracting.  This gives an effective tool for analyzing the contracting boundary.  As an illustration, we apply these techniques to an example of two right-angled Coxeter groups whose quasi-isometry classes are not distinguished by any of the standard invariants.  We show that their contracting boundaries are not homeomorphic, hence the groups are not quasi-isometric.

Other notions of boundaries for CAT(0) cube complexes have been introduced by Roller \cite{roller}, Guralnik \cite{guralnik},  Nevo-Sageev \cite{nevosageev}, Hagen \cite{Hagen} and Behrstock-Hagen \cite{behrstockhagen}.   It would be interesting to better understand the relationship between these boundaries.  In particular, the Nevo-Sageev boundary seems closely related to the contracting boundary in the case of a rank one cube complex, although it is defined as a measure space, not a topological space.

The outline of the paper is as follows.  In Section \ref{sec:hyperbolictype}, we consider various notions of hyperbolic type geodesics and prove their equivalence.  In Section \ref{sec:contractingboundary}, we introduce the contracting boundary of a CAT(0) space and establish its key properties.  In particular, we prove that  the homeomorphism type of the contracting boundary is a quasi-isometry invariant.  In Section \ref{sec:cubecomplexes}, we specialize to the case of CAT(0) cube complexes and give a combinatorial condition for a geodesic to be contracting.  Finally, in Section \ref{application}, we apply these results to some examples of right-angled Coxeter groups. 

The first author would like to thank the  Forschungsinstitut f\"ur Mathematik in Zurich for their hospitality during development of this paper.  Both authors would like to thank Mladen Bestvina, Moon Duchin, and Koji Fujiwara for helpful conversations.

\section{Hyperbolic type geodesics}\label{sec:hyperbolictype}
\subsection{Background}\label{sec:back}

A \emph{geodesic} in a metric space $X$ is an isometric embedding of a (finite or infinite) interval into $X$.  A geodesic metric space is one in which any two points are connected by a geodesic.  A
CAT(0) space is a geodesic metric space defined by the property that geodesic triangles are no ``fatter'' than the corresponding comparison triangles in Euclidean space. 
We refer the reader to \cite{bridsonhaefliger} for a precise definition and basic properties of CAT(0) spaces.  

The following lemma describes two fundamental properties of CAT(0) spaces that will be used frequently in this paper, see \cite[Section II.2]{bridsonhaefliger} for details. 

\begin{lemma} Let $X$ be a CAT(0) space.
\begin{enumerate} \label{lem:cat}
\item[C1:] (Unique geodesics).  $\forall x,y \in X$ there is a unique geodesic connecting $x$ and $y$.  We denote this segment by $[x,y]$.
\item [C2:] (Projections onto convex subsets). Let $C$ be a convex subset, complete in the induced metric, then there is a well-defined distance non-increasing nearest point projection map $\pi_{C}\co X \ra C.$  In particular, $\pi_{C}$ is continuous.  
\item [C3:] (Convexity). Let $c_{1}\co [0,1] \ra X$ and $c_{2}\co [0,1] \ra X$ be any pair of geodesics parameterized proportional to arc length.  Then the following inequality holds for all $t\in [0,1]:$
$$d(c_{1}(t),c_{2}(t)) \leq (1-t) d(c_{1}(0),c_{2}(0)) + td(c_{1}(1),c_{2}(1))  $$
\end{enumerate}
\end{lemma}  

The following notions will also play a central role in this paper.

\begin{definition}[quasi-isometry; quasi-geodesic] \label{defn:quasi} A map $f:X \ra Y$ is called a \emph{(K,L)-quasi-isometric embedding} if $\forall s,t \in X$  the following inequality holds:
\begin{equation} \label{eq:qi} \frac{1}{K}\, d_{X}(s,t) -L \leq  d_{Y}(f(s),f(t)) \leq Kd_{X}(s,t)+L.
\end{equation} 
 If, in addition, $f$ satisfies
\begin{equation} \forall y \in Y, \, \exists x \in X  \textrm{ such that } d_{Y}(f(x),y) < L,
\end{equation} 
then $f$ is called a \emph{(K,L)-quasi-isometry}.  The special case of a quasi-isometric embedding where the domain is a connected interval in $\R$ (possibly all of $\R$) is called a \emph{(K,L)-quasi-geodesic.} 

Given a quasi-isometry $f:X \ra Y$,  there exists a \emph{quasi-inverse} $g:Y \ra X$, which is itself is a quasi-isometry such that there exists a constant $C,$ depending only on $K,L$,  with the property  that for all $x \in X, y \in Y$,
$$d_{X}(x,gf(x)) \leq C\, \textrm{ and }\,  d_{Y}(y,fg(y) \leq C.$$
 \end{definition}

\subsection{Contracting and Morse Geodesics}

We are interested in geodesics which behave similarly to geodesics in a hyperbolic space.  A key property of hyperbolic geodesics is the contracting property.  
The following notion of contracting geodesics can be found for example in \cite{bestvinafujiwara}, and has its roots based in a slightly more general notion of $(a,b,c)$--contraction found in \cite{mm1} where it serves as a key ingredient in the proof of the hyperbolicity of the curve complex.

\begin{definition}[contracting geodesics] Given a fixed constant $D,$ a geodesic $\gamma$ is said to be \emph{D--contracting} if $\forall x,y \in X,$ 
\[
d_{X}(x,y)<d_{X}(x,\pi_{\gamma}(x))  \implies d_{X}(\pi_{\gamma}(x),\pi_{\gamma}(y))<D.
\]
We say $\gamma$ is \emph{contracting} if it is $D$-contracting for some $D$.  Equivalently, any metric ball $B$ not intersecting $\gamma$ projects to a segment of length $< 2D$ on $\gamma$.
\end{definition}

 In this section we will give several equivalent characterizations of contracting geodesics.  The contracting boundary introduced in the next section   will consist  precisely of such contracting geodesic rays.   The various equivalent characterizations will be used to in prove key properties of the contracting boundary.

The first characterization is the notion of a Morse (quasi-)geodesic which has roots in the classical paper \cite{morse}.  For any subset $A \subset X$ and constant $r>0$, let $N_r(A)$ denote the $r$-neighborhood of A.  Recall that two subspaces $A,B \subset X$ have \emph{Hausdorff distance} at most $r$ if $A \subset N_r(B)$ and $B \subset N_r(A)$.

\begin{definition}[Morse quasi-geodesics] A (quasi-)geodesic $\gamma$ is called M--\emph{Morse} if for any constants $K\geq1,L\geq0,$ there is a constant $M=M(K,L),$ such that for every $(K,L)$-quasi-geodesic $\sigma$ with endpoints on $\gamma,$ we have $\sigma \subset N_{M}(\gamma).$   \end{definition} 

The following properties of Morse quasi-geodesics are easily verified.  

\begin{lemma}\label{lem:constants}  Let $\gamma$ be an $M$-Morse quasi-geodesic in a CAT(0) space $X$.
\begin{enumerate}
\item If  $\rho$ is a quasi-geodesic whose Hausdorff distance from $\gamma$ is at most $C$, then $\rho$ is $M'$-Morse where $M'$ depends only on $M$ and $C$.
\item If $Y$ is a geodesic metric space and $f: X \to Y$ is a $(\lambda,\epsilon)$-quasi-isometry, then $f\circ\gamma$ is $M''$-Morse where $M''$ depends only on $\lambda,\epsilon$ and $M$.
\item If $\gamma$ is a $(K,L)$-quasi-geodesic, there exists $C$ depending only on $M, K, L$ such that for any two points  $x=\gamma(t)$ and $y=\gamma(t')$  on $\gamma$,  the geodesic $[x,y]$ has Hausdorff distance at most $C$ from $\gamma([t,t'])$.  
\end{enumerate}
\end{lemma}

\begin{proof}  (1) This is an easy exercise which we leave to the reader.

(2)  Let $g : Y \to X$ be a quasi-inverse of $f$.  Then $g\circ f\circ \gamma$ has Hausdorff distance at most $C$ from $\gamma$ for some $C$ depending only on $\lambda,\epsilon$, so by part (1), it is $M'$--Morse where $M'$ depends only on $M,\lambda,\epsilon$.  Suppose $\beta$ is a $(K,L)$-quasi geodesic in $Y$ between two points on $f \circ \gamma$. Then $g \circ \beta$ is a $(K',L')$-quasi-geodesic between two points on $g \circ f \circ \gamma$ where $K',L'$ depend only on $K,L,\lambda,\epsilon$, so $g \circ \beta$ lies in the $M'(K',L')$-neighborhood of $g \circ f\circ \gamma$. It follows that $\beta$ lies in the $\lambda(M'(K',L')  +\epsilon)$-neighborhood of $f \circ \gamma$.  Setting 
$M''(K,L)=\lambda(M'(K',L')  +\epsilon)$, we conclude that $f \circ \gamma$ is $M''$-Morse.

(3)  The proof of this statement follows the proof of Theorem 1.7 in \cite{bridsonhaefliger} III.H.  Set $M_0=M(1,0)$.
Since $\gamma$ is $M$-Morse and $[x,y]$ is geodesic, $[x,y] \subset N_{M_0}(\gamma([t,t']))$.   Thus, it suffices to find $C$ such that $\gamma([t,t']) \subset N_C([x,y])$.    By Lemma 1.11 of \cite{bridsonhaefliger} III.H, we may assume without loss of generality  that $\gamma([t,t'])$ is ``tame", that is, it is continuous and for any subinterval $[s,s'] \subseteq [t,t']$, 
$$ length(\gamma([s,s'])  \leq k_1 d(\gamma(s),\gamma(s')) + k_2$$
where $k_1,k_2$ depend only on $\lambda,\epsilon$. 

  If  $\gamma([t,t']) \subset N_{M_0}([x,y])$ we are done.  If not, consider a maximal segment $[s,s'] \subset [t,t']$ such that $\gamma([s,s'])$ lies outside $N_{M_0}([x,y])$. Every point of $[x,y]$ lies within $M_0$ of some point on $\gamma([t,t'])$,  so by continuity, there exists a point $z \in [x,y]$, such that $z$ lies within $M_0$ of two points, $\gamma(r)$ and $\gamma(r')$, with 
$r \in [t,s], r' \in [s',t']$.  In particular, $d(\gamma(r),\gamma(r')) \leq 2M_0$. By the tameness condition, it follows that $length(\gamma([r,r']) \leq 2k_1M_0 +k_2$.  Hence $\gamma([r,r'])$ lies within $M_0+k_1M_0 + k_2/2$ of $z$.  Taking $C=M_0+k_1M_0 + k_2/2$, we conclude that .$\gamma([t,t']) \subset N_C([x,y])$. 
\end{proof}

To prove the equivalence of contracting geodesics and Morse geodesics, we will need to understand quasi-geodesics of a particular form considered in the next two lemmas.

\begin{lemma}\label{lem:quasigeodesic} Let $X$ be a CAT(0) space. 
For any triple of points $x,y,z \in X,$ the concatenated path $$\phi=[x,\pi_{[y,z]}(x)] \cup [\pi_{[y,z]}(x),z],$$ is a (3,0) quasi-geodesic.  
\end{lemma}
\begin{proof}
We must show that $\forall u,v \in \phi,$ the (3,0)--quasi-isometric inequality of Equation (\ref{eq:qi}) is satisfied.  Since $\phi$ is a concatenation of two geodesic segments, without loss of generality we can assume $u \in  [x,\pi_{[y,z]}(x)], v \in [\pi_{[y,z]}(x),z].$  Since $u \in[x,\pi_{[y,z]}(x)]$ it follows that $\pi_{[y,z]}(u)=\pi_{[y,z]}(x),$ and hence $d(u,\pi_{[y,z]}(x))\leq d(u,v).$  Let $d_{\phi}(u,v)$ denote the distance along $\phi$ between $u$ and $v.$  
Then, the following inequality completes the proof:
\begin{eqnarray*}
d(u,v) \leq d_{\phi}(u,v) &=& d(u,\pi_{[y,z]}(x)) + d(\pi_{[y,z]}(x),v) \\
& \leq& d(u,\pi_{[y,z]}(x)) + \left(d(u, \pi_{[y,z]}(x)) + d(u,v) \right) \\ 
 & \leq& 3d(u,v)
\end{eqnarray*} 
\end{proof}

Next, we consider a concatenation of three geodesics segments.  We will show that by cutting off ``corners", we can obtain a  quasi-geodesic with controlled quasi-constants.

Let $\gamma$ be a geodesic, and let $x,y \in X$ be two points not on $\gamma$.  As in Figure \ref{fig:cases},  set $D=d(\pi_{\gamma}(x),\pi_{\gamma}(y))$  and let $a,b,c$ be constants such that
\begin{equation}\label{abc}
aD = d(x,\pi_{\gamma}(x)) \quad  bD = d(x,y)  \quad cD = d(y,\pi_{\gamma}(y)).
\end{equation}
  For any point $z$ in $[x,y]$, $d(z, \pi_{\gamma}(z)) \geq aD-d(z,x)$ and 
 $d(z, \pi_{\gamma}(z)) \geq cD-d(z,y)$, so the distance from $z$ to $\gamma$ is at least the average of these two quantities, namely, 
 \begin{equation}\label{fig inequality}
 d(z, \gamma) \geq \frac{D(a+c-b)}{2}.
 \end{equation}

\begin{figure}[htpb] 
\centering
\includegraphics[height=7 cm]{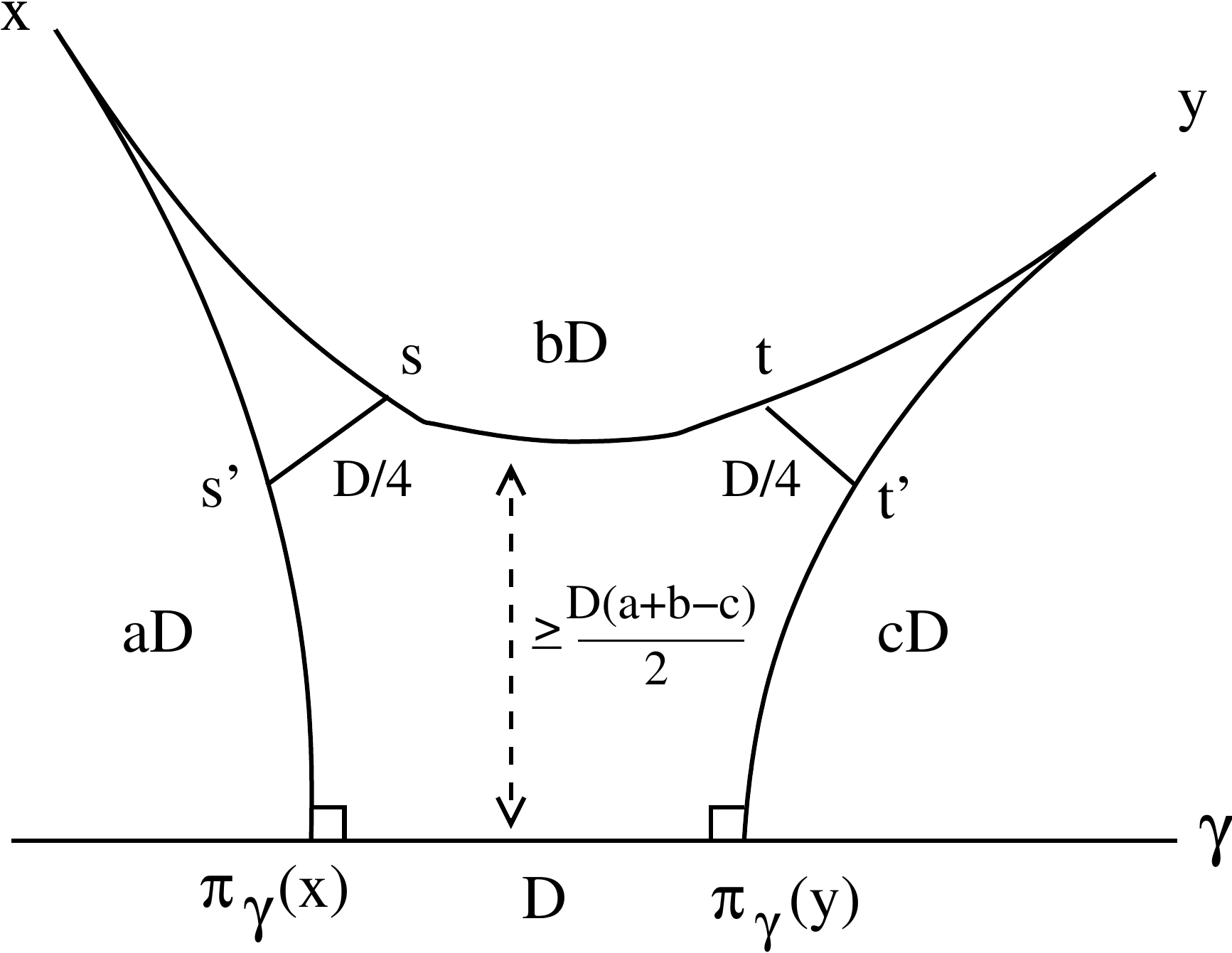}
\caption{Illustration of Lemma \ref{lem:genquasi}.}\label{fig:cases}
\end{figure}       

 Note that by property [C2] of Lemma \ref{lem:cat}, $b \geq 1.$  By property [C3] of the lemma, the function $\rho(z)=d(z, [x,\pi_{\gamma}(x)])$ is convex.  Thus, as $z$ goes from $x$ to $y$, it is strictly increasing once $\rho(z) > 0$.   Hence there is a unique point  $s \in [x,y]$ such that $d(s, [x,\pi_{\gamma}(x)])=\frac{D}{4}.$  Let $s'$ be the projection of $s$ on $[x,\pi_{\gamma}(x)].$  Similarly, there is a unique point $t \in [x,y]$  such that $d(t, [y,\pi_{\gamma}(y)])=\frac{D}{4}.$  Let $t'$ be the projection of $t$ on $[y,\pi_{\gamma}(y)].$  
Note that the projection of $[s,t]$ on $\gamma$ has length at least $\frac{D}{2}$, hence $[s,t]$ also has length at least 
$\frac{D}{2}.$



\begin{lemma} \label{lem:genquasi}  With notation as above, set $K=4(a+b+c)$.
Then the concatenation $$\phi=[\pi_{\gamma}(x),s'] \cup [s',s] \cup [s,t] \cup [t,t'] \cup [t',\pi_{\gamma}(y)],$$ is a $(K, 0)$-quasi-geodesic.  
\end{lemma}

\begin{proof} We will show that $\forall w,z \in \phi,$ the $(K, 0)$--quasi-isometric inequality is satisfied.  Since $\phi$ is a concatenation of geodesics, without loss of generality we can assume $w,z$ belong to different geodesic segments within $\phi.$    In the case where $w$ and $z$ belong to adjacent segments, it follows from Lemma \ref{lem:quasigeodesic}  and the fact that $K >4b>3$ that the (K,0)--quasi-isometric inequality holds.    Hence, to complete the proof,  it suffices to consider the case in which $w,z$ are separated by at least one of the segments $[s',s], [s,t],$ or $[t,t']$.

If one of the points  $w,z$ lies on $[s,t]$ and the other on $[\pi_{\gamma}(x),s'] $ or $[t',\pi_{\gamma}(y)],$ then by construction, $d(w,z )>\frac{D}{4}$.  Otherwise, $w,z$ lie on opposite  sides of $[s,t]$, hence their projections on $\gamma$ have distance at least $\frac{D}{2}$, so $d(w,z) \geq \frac{D}{2}$.  In either case, $D < 4 \,d(w,z),$ so we have
$$ d(w,z) \leq d_{\phi}(w,z) \leq (a+b+c)D < 4(a+b+c) \, d(w,z) = K\,  d(w,z). $$
\end{proof}

Applying Lemma \ref{lem:genquasi} in the case where $\gamma$ is a Morse geodesic, we have the following corollary.

\begin{corollary} \label{cor:genquasi}  Let $\gamma$ be an $M$-Morse geodesic and $x,y \in X$. With notation as above, set $K=4(a+b+c)$,  $N=\frac{1}{2}(a+c-b)$.  
If $N >0$, then $D \leq \frac{1}{N}M(K,0)$.   
 \end{corollary}
 
 \begin{proof}
As observed in equation \ref{fig inequality} above,  $d([x,y],\gamma) \geq DN.$  Since $[s,t] \subset [x,y],$ in particular $d([s,t],\gamma) \geq DN.$  On the other hand, $[s,t]$ is a segment of the $(K,0)$-quasi-geodesic $\phi$ and hence must stay within the $M(K,0)$ neighborhood of $\gamma$.    Combining the inequalities, the corollary follows.
\end{proof}  

We are now ready to prove the equivalence of the contracting and Morse conditions.  We remark that in \cite{sultanmorse}, the first author proved that these two notions are  equivalent, but without the explicit control on the constants.  This control on constants will be essential to our understanding of contracting boundaries.

\begin{theorem}
\label{theorem:equivalence}
Let X be CAT(0) and $\gamma \subset X$ a geodesic.  Then the following are equivalent:
\begin{enumerate}
\item $\gamma$ is $D$--contracting, 
\item $\gamma$ is $M$--Morse
\end{enumerate}
Moreover,  the Morse function $M$ is determined by the constant $D$ and vice versa.   
\end{theorem}

\begin{proof}  
The fact that $(1) \implies (2)$ with explicit constants, is the well known ``Morse stability lemma.''  For a proof see for instance \cite{algomkfir} or \cite{sultanmorse}. 

We will prove that $(2) \implies (1)$ with a bound on $D$ determined by $M$.  Fix $x,y \in X$ such that $d(x,y) < d(x, \pi_{\gamma}(x)).$  Set $P = d( \pi_{\gamma}(x)),  \pi_{\gamma}(y))$, 
$$
A = d(x, \pi_{\gamma}(x)), \quad B = d(x,y), \quad C= d(y, \pi_{\gamma}(y)).
$$
 Without loss of generality, 
we may assume $A \geq C$ (if not, reverse the labels on $x$ and $y$), and by assumption, $A > B$.  
We consider three cases.

{\bf Case (1):}  $A \geq C \geq 2P$

Let $x'$ be the point on $[x, \pi_{\gamma}(x)]$ at distance $2P$ from $\pi_{\gamma}(x)$ and let $y'$ 
be the point on $[y, \pi_{\gamma}(y)]$ at distance $2P$ from $\pi_{\gamma}(y)$.  We claim that $B'=d(x',y') \leq 3.9 P$.
To see this, consider the two triangles $\Delta_1 = \Delta(x,  \pi_{\gamma}(x)),  \pi_{\gamma}(y))$ and $\Delta_2 =
\Delta(x,  y,  \pi_{\gamma}(y)))$ and let $\overline\Delta_1$ and $\overline\Delta_2$ denote the comparison triangles in Euclidean space.  

Let $M=d(x,  \pi_{\gamma}(y))$.  The angle of the triangle $\overline\Delta_1$ at the vertex $\overline{ \pi_{\gamma}(x)}$ is at least $\frac{\pi}{2}$, so $M>\sqrt{A^2+P^2} > A>B$.  It follows that the angle in $\overline\Delta_2$ at the vertex $\overline{ \pi_{\gamma}(y)}$ must be less than $\frac{\pi}{2}$.  Now let $z'$ be the point on $[x, \pi_{\gamma}(y)]$ at distance $2P$ from $\pi_{\gamma}(y)$.  An exercise in Euclidean geometry shows that the points $\overline x', \overline z', \overline y'$ corresponding to $x',z',y'$ satisfy $d(\overline x',\overline z') \leq P$ and $d(\overline z',\overline y') \leq 2\sqrt{2} P $.  Hence,
$$d(x',y') \leq d(\overline x',\overline z')+d(\overline z',\overline y') \leq P + 2\sqrt{2} P < 3.9 P.$$

Replacing $x,y$ by $x',y'$ and setting $A' = d(x', \pi_{\gamma}(x)), B' = d(x',y'),  C'= d(y', \pi_{\gamma}(y))$, we now have $A' -B' +C' \geq 2P -3.9P + 2P = 0.1P$ and $A' + B' +C' \leq 8P$, so Corollary \ref{cor:genquasi} guarantees that 
$P  < 20 M(32,0)$.  

\begin{figure}[htpb] 
\centering
\includegraphics[height=6 cm]{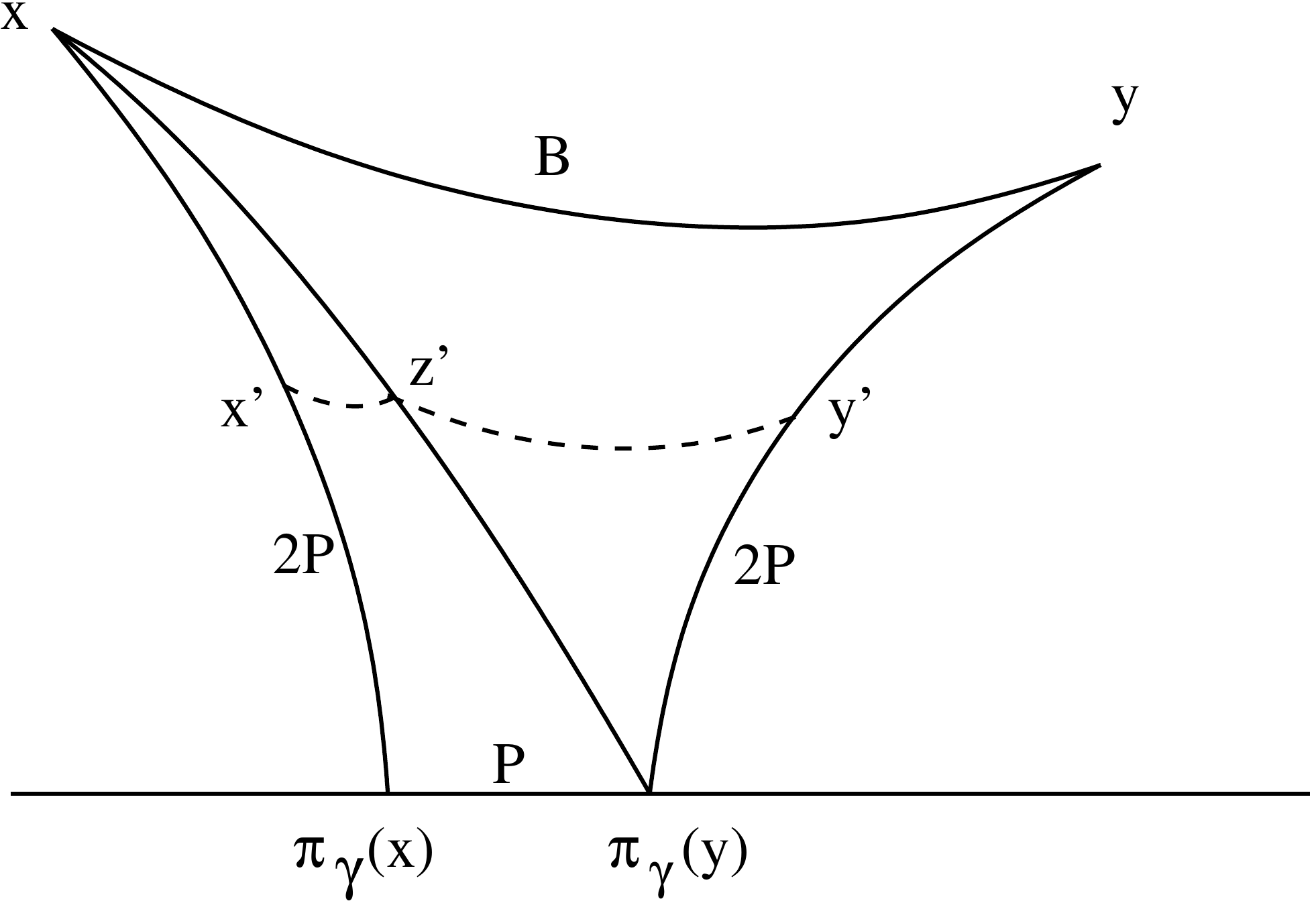}
\caption{Case (1)} \label{fig:equiv-case1}
\end{figure}      

{\bf Caes (2):}  $A \geq 2P >  C$

Consider the function on $[x,y]$ defined by $f(z)=d(z,\pi_{\gamma}(z)) - 2 d(\pi_{\gamma}(x), \pi_{\gamma}(z))$.  
Note that $f(x)= A > 0$ and $f(y)=C-2P < 0$.  The function is continuous, so there exists a point $z$ on $[x,y]$ with
$f(z)=0$.  Setting 
$$P'= d(\pi_{\gamma}(x), \pi_{\gamma}(z)), \quad B'=d(x,z), \quad  C'=d(z,\pi_{\gamma}(z)),$$
  we have $A > B'$ and  $A \geq 2P'=C'$, so 
Case (1) applied to $x,z$ shows that $P'  < 20 M(32,0)$.

Now observe that $A < d(x,z) + d(z, \pi_{\gamma}(z)) + d(\pi_{\gamma}(z),\pi_{\gamma}(x)) = B' +  3P'$,
so $ d(z,y) = B-B' < A-B' < 3P'$.  Hence,
\begin{eqnarray*}
P &=&  d(\pi_{\gamma}(x), \pi_{\gamma}(z)) + d(\pi_{\gamma}(z), \pi_{\gamma}(y))\\
 & \leq & d(\pi_{\gamma}(x), \pi_{\gamma}(z)) + d(z,y) \\
 & < &  4P' < 80 M(32,0).
\end{eqnarray*} 


{\bf Case (3):}  $ 2P > A \geq  C$

In this case we have $A+B+C \leq 3A  < 6P$,  and $A-B+C >C$,  so we need 
only show that  $C$ is bounded below.  Let $M = d(x, \pi_{\gamma}(y))$ as in Case (1).  Then 
$$\sqrt{A^2 + P^2} \leq M  \leq B+C \leq A+C.$$ 
Letting $a=A/P$ and $c=C/P$, this inequality can be rewritten as $c \geq \sqrt{a^2 + 1} - a$.  
The function  $\rho(x)= \sqrt{x^2 + 1} - x$ is a decreasing function so $c \geq \rho(a) \geq \rho(2) = \sqrt{5} -2 > 0.2$.  By Corollary \ref{cor:genquasi}, we conclude that $P < 10 M(24,0)$.

Setting $D=80M(32,0)$, we conclude that $P=  d( \pi_{\gamma}(x)),  \pi_{\gamma}(y))< D$ in all three cases.
\end{proof}

This theorem has an important corollary.

\begin{corollary}\label{cor:DtoD'}  Let $X,Y$ be CAT(0) spaces with $Y$ complete, and let  $f : X \to Y$ be a $(\lambda,\epsilon)$-quasi-isometry.  Then for any $D$-contracting geodesic ray $\gamma$ in $X$ based at $x_0$, $f \circ \gamma$ stays bounded distance from a unique geodesic ray $\beta$ based at $f(x_0)$.  Moreover, $\beta$ is $D'$-contracting where $D'$ depends  only on $\lambda,\epsilon$ and $D$.
\end{corollary}

\begin{proof}   By Theorem \ref{theorem:equivalence}, $\gamma$ is $M$-Morse, with $M$ depending only on $D$.  By Lemma \ref{lem:constants}(2), $f \circ \gamma$ is $M''$-Morse with $M''$ depending only on $\lambda,\epsilon$ and $M$.  Let  $\beta_n$ be  the geodesic segment  from $f(x_0)=f(\gamma(0))$ to $f(\gamma(n))$, $n \in \N$.  Since $f \circ \gamma$ is a $(\lambda, \epsilon)$-quasi-geodesic,  Lemma \ref{lem:constants}(3) implies that there exists $C$, depending only on  $M'',\lambda,\epsilon$, such that each $\beta_n$ lies Hausdorff distance at most  $C$ from $f\circ \gamma|_{[0,n]}$.   Thus, restricted to $[0,n]$,  the segments $\beta_m, m\geq n$ all lie Hausdorff distance at most $2C$ from each other.  The CAT(0) thinness condition and the completeness of $Y$ then guarantee that the sequence $(\beta_n)$ converges to a unique geodesic ray $\beta$ lying Hausdorff distance at most $C$ from $f\circ \gamma$.  By Lemma  \ref{lem:constants}(1), it follows that  $\beta$ is $M'$-Morse where $M'$ depends only on $M'', C$.  Finally, applying Theorem \ref{theorem:equivalence} once again, we conclude that $\beta$ is $D'$ contracting where $D'$ depends only on $M'$.
\end{proof}

\subsection{Thin triangle conditions}

In addition to the Morse condition, there are several other conditions equivalent to the contracting property.
These illustrate the principle that contracting geodesics behave like hyperbolic geodesics and will play a role in applications which follow later in this paper.  The first of these properties is a thin triangle condition.  
 
Let $\alpha$ be a geodesic.  We say $\alpha$ satisfies  \emph{thin triangle condition (i)} if there exists $\delta$ such that for all $x \in X,y \in \alpha$, we have
$$d(\pi_{\alpha}(x), [x,y]) < \delta.$$  
We say $\alpha$ satisfies \emph{thin triangle condition (ii)} if there exists $\delta$ such that for all geodesic triangles $\Delta(x,y,z)$ with $[y,z] \subset \alpha$, every point $w \in [y,z]$ satisfies 
$$d(w, [x,y] \cup[x,z]) < \delta.$$   

\begin{lemma}  
\label{lem:thinequivalent} The two thin triangle conditions are equivalent (though the $\delta$ required may be different).
\end{lemma}
\begin{proof}  Thin condition (i) implies thin condition (ii):  Let $x \in X$ and $[y,z] \subset \gamma$.  Applying condition (i) to $x,y$ gives $d(\pi_{\gamma}(x),[x,y]) < \delta.$  The CAT(0) thinness condition then implies that $d(w, [x,y]) < \delta$ for all $w \in [\pi_{\gamma}(x),y]$.  The same argument applied to $x,z$ shows that $d(w, [x,z]) < \delta$ for all $w \in [\pi_{\gamma}(x),z]$.

Thin condition (ii) implies thin condition (i):  Set $z=\pi_{\gamma}(x)$.  The Euclidean comparison triangle for $\Delta(x,y,z)$ has angle at least $\frac{\pi}{2}$ at $z$, so for any $w \in [z,y]$, $d(w, [x,z])=d(w,z)$.  In particular, choosing $w$ such that $\delta < d(w,z)< 2\delta$,   condition (ii) implies that  $d(w, [x,y]) < \delta$, so $d(z,[x,y]) < 3\delta$.
\end{proof}

\begin{definition}[slim geodesic] \label{def:slim}A geodesic $\gamma$ is said to be \emph{$\delta$--slim} if $\gamma$ satisfies thin triangle condition (i) with the constant $\delta.$  When the constant $\delta$ is not relevant we will omit it from the notation and simply say $\gamma$ is \emph{slim}.  Note that by Lemma \ref{lem:thinequivalent}, $\gamma$ is slim if and only if $\gamma$ satisfies thin triangle condition (ii).
\end{definition}   
 
Thin triangle condition (i) is used, for example by Bestvina-Fujiwara in \cite{bestvinafujiwara} where it is shown that 
if  $\gamma$ is a $D$--contracting geodesic in a CAT(0) space,  then $\gamma$ is $(3D+1)$-slim. As we will see below, the converse is also true, that is,  slimness implies contracting.

\subsection{Lower divergence} 
 
The last notion of hyperbolic type is based on a variation of divergence.
\begin{definition}[lower divergence]  Let $\gamma$ be a quasi-geodesic.  For any $t>r>0$, let $\rho_{\gamma}(r,t)$ denote the infimum of the lengths of all paths from $\gamma(t-r)$ to $\gamma(t+r)$ which lie outside the open ball of radius $r$ about $\gamma(t)$. Define the \emph{lower divergence} of $\gamma$ to be the growth rate of the following function:
\[
ldiv_\gamma(r) = \inf_{t > r} \rho_{\gamma}(r,t).
\] 
 \end{definition}

The key difference between lower divergence and the more standard notion of divergence (see for example \cite{drutumozessapir}) is that in the standard notion, one considers only balls with some fixed center  $\gamma(t)$, whereas for lower divergence, we allow the center to vary over all of $\gamma$.   This flexibility is essential for working with geodesic rays (as opposed to bi-infinite geodesic lines), but even in the case of geodesics lines, the two notions are different.  Consider, for example, two flats joined at a single point $x$. Any geodesic line passing through $x$ will have infinite divergence, but linear lower divergence.  Or consider the space $X$ formed by slitting open the Poincar\'e disc  along the positive half of a geodesic $\alpha$ and inserting a Euclidean sector (Figure \ref{fig:pacman}).  The image of $\alpha$ in $X$ has exponential divergence but linear lower divergence.  In the case of a periodic geodesic, however, the two notions are equivalent.

\begin{figure}[htpb] 
\centering
\includegraphics[height=5 cm]{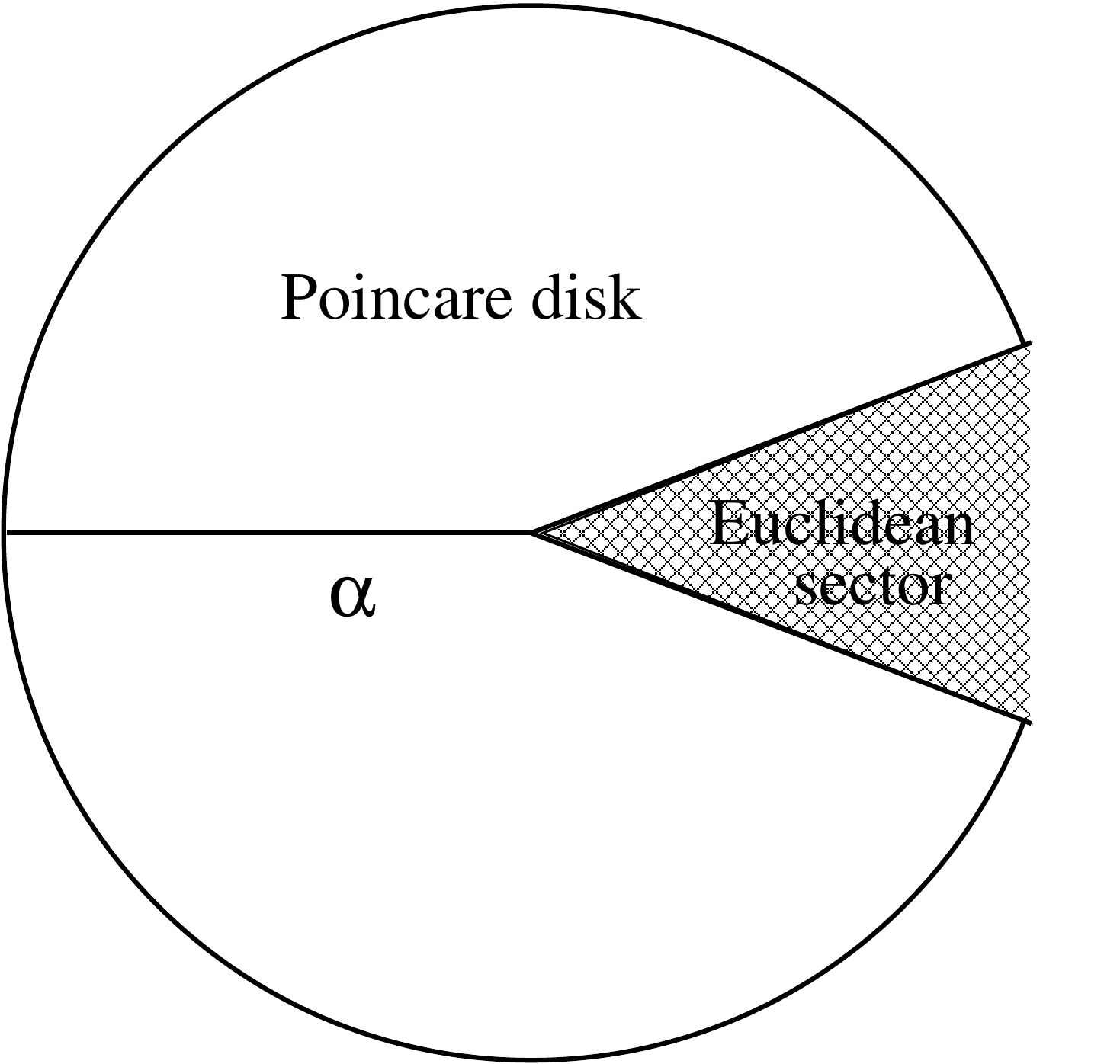}
\caption{divergence vs. lower divergence} \label{fig:pacman}
\end{figure}      

We now show that all of these notions are equivalent.
 
\begin{theorem} \label{theorem:equivalenceexpanded}  Let $X$ be a CAT(0) space and 
let $\alpha \subset X$ be a geodesic ray or line.  Then the following are equivalent:
\begin{enumerate}
\item  $\alpha$ is contracting, 
\item  $\alpha$ is Morse,
\item  $\alpha$ is slim,
\item  $\alpha$ has superlinear lower divergence, 
\item  $\alpha$ has at least quadratic lower divergence.

\end{enumerate}
\end{theorem}

\begin{example}[Teichm\"uller space]  
Recall that Teichm\"uller space equipped with the WP metric is CAT(0).  Using Theorem \ref{theorem:equivalence}, we can characterize all contracting quasi-geodesics in the space.  Considering the literature, it is apparent that the study of contracting geodesics in Teichm\"uller space is of utmost interest, both for identifying interesting phenomena among geodesics and for enhancing understanding of the space as a whole.

In \cite{mm2} a 2-transitive family of quasi-geodesics in Teichm\"uller space called hierarchy paths are introduced.  In \cite{sultanthesis} it is shown that a hierarchy path is Morse if and only if there is a uniform bound on the distance traveled in all component domains whose complement in the surface contains a connected essential subsurface with complexity at least one (or equivalently contains a connected subsurface such that the Teichm\"uller space of the subsurface is nontrivial).  It follows that a geodesic is similarly Morse if and only if there is a uniform bound on the subsurface projection distance to any essential subsurface whose complement in the surface contains a connected essential subsurface of complexity at least one.  In light of Theorem \ref{theorem:equivalenceexpanded}, the aforementioned characterization of Morse geodesics provides an equivalent characterization for contracting geodesics, as well as each of the hyperbolic type geodesics considered in Theorem \ref{theorem:equivalenceexpanded}.  It should be noted that for the once punctured torus or the four times punctured sphere, contracting geodesics are precisely geodesics with so called \emph{bounded geometry} studied in \cite{bmm2}.  More generally, for larger surfaces, the family of contracting geodesics  includes the family of geodesic with bounded geometry as a proper subset.
\end{example}

\begin{proof}[Proof of Theorem \ref{theorem:equivalenceexpanded}] We have already proved the equivalence of (1) and (2).  We will prove the equivalence of
$(1), (3), (4), (5)$. Note that  $(5) \implies (4)$ is obvious.

\medskip
\noindent  $(4) \implies (3)$. 
  Suppose thin triangle condition (ii) fails.  Then for any $r>0$, there exists 
a triangle $\Delta(x,y,z)$ with $[y,z] \subset \alpha$, and a point $w \in [y,z]$ such that $d(w, [x,y] \cup[x,z]) \geq r$.  Moving $x$ closer to $[y,z]$ if necessary, we may assume that 
$d(w, [x,y] \cup[x,z]) = r$.  Moreover, we may choose $w$ so that $d(w, [x,y]) =(w, [x,z] )=r$.  See Figure \ref{fig:4to3}.

\begin{figure}[htpb] 
\centering
\includegraphics[height=6 cm]{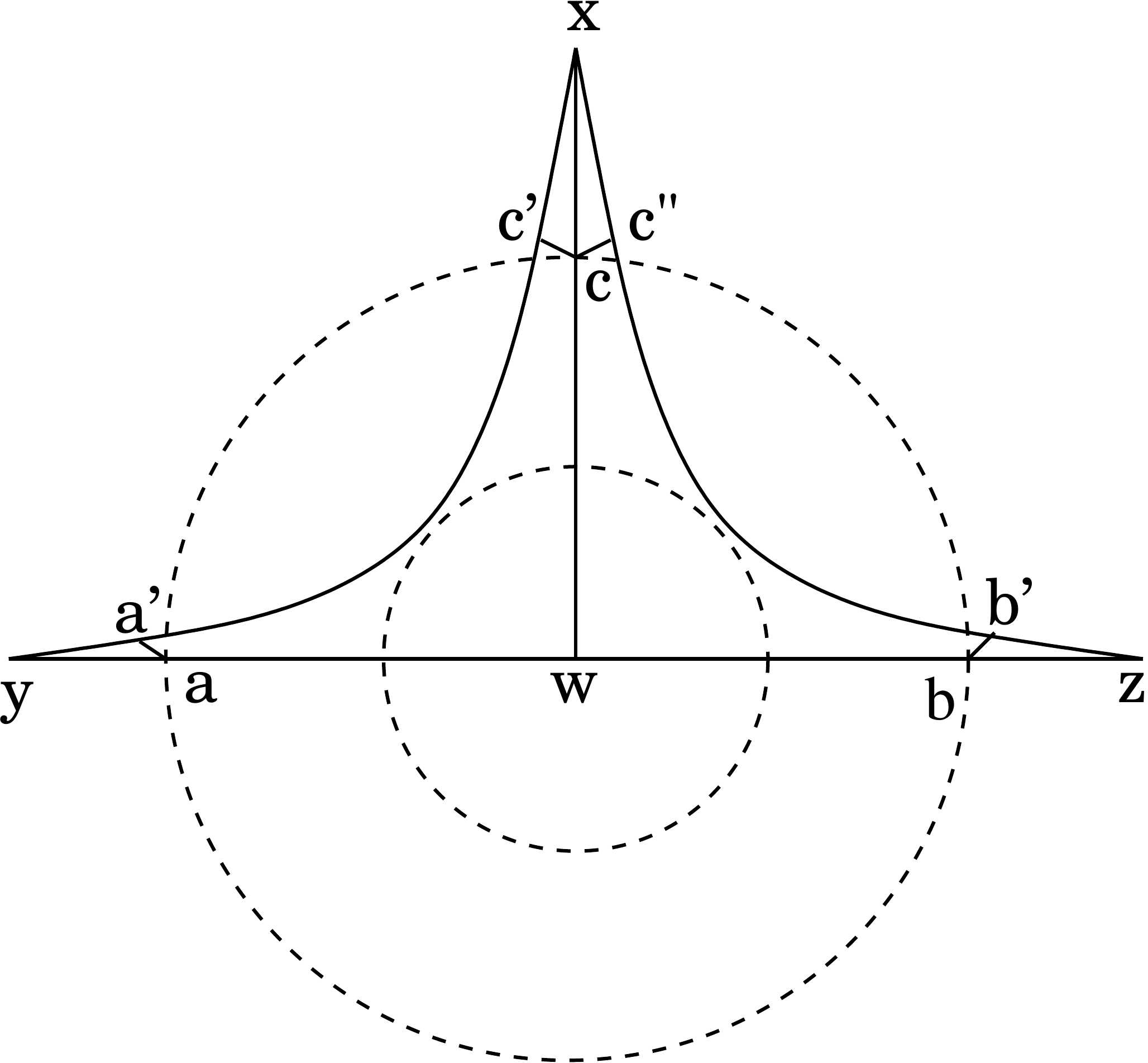}
\caption{$(4) \implies (3)$} \label{fig:4to3}
\end{figure}

 Let $a \in [y,w]$ be the point at distance $2r$ from $w$ (or if $d(w,y) < 2r$, take $a_1=y$).  
 Let $a'$ be the projection of $a$ on $[x,y]$. It follows from the convexity of the metric that  
$ d(a,a')=d(a,[x,y]) \leq d(w, [x,y])=r$.  Likewise, if  $b \in [z,w]$ is the point at distance $2r$ from $w$ (or 
$b=z$ if  $d(w,z) < 2r$), and $b' $ is its projection on $[x,z]$, then $d(b,b') \leq r$.

Finally, let $c \in [x,w]$ be the point at distance $2r$ from $w$ (or 
$c=x$ if  $d(w,x) < 2r$). Then the projections $c'$ and $c''$ on on $[x,z]$ and $[x,y]$ satisfy 
$d(c,c') \leq r$ and $d(c,c') \leq r$.  

Now say $w=\alpha(t)$.  Then $w_1=\alpha(t-r)$ and $w_2=\alpha(t+r)$ both lie in $[a,b]$. Consider the path $\gamma$ from $w_1$ to $w_2$ formed by the segments
$$ [w_1,a] \cup [a,a'] \cup [a',c'] \cup [c',c] \cup [c,c''] \cup [c'',b'] \cup [b',b] \cup [b,w_2]. $$
By construction, this path lies outside the open ball of radius $r$ about $w$.  To compute the length of $\gamma$, note that since $a',c'$ are the projections of $a,c$ on $[x,y]$,
$d(a',c') \leq d(a,c) \leq d(x,w)+d(w,c) \leq 4r$, and likewise $d(b',c'') \leq 4r$.
Thus adding the lengths of all the segments in $\gamma$ gives $|\gamma| \leq 14r$.
It follows that $\rho_\alpha(r,t) \leq 14r$, contradicting our assumption that  $ldiv_\alpha$ is super-linear.
 
 \medskip
\noindent  $(3) \implies (1)$.  Assume $\alpha$ satisfies the thin triangle condition (i) with constant $\delta$.  Let $B=B(x,r)$ be a ball in $X$ not intersecting $\alpha$.  Let 
$y \in B$ and let $x',y'$ denote the projections of $x,y$ on $\alpha$. Set $A=d(x,x')$  and note that $A=d(x,\alpha) \geq r$.  By the thin triangle condition (applied at $y$), 
there exists a point $z \in [x',y]$ such that  $d(y',z) < \delta$.  By the CAT(0) condition,
$d(x,z) \leq max\{d(x,x'),d(x,y)\}=A$, so 
$$d(x,y') \leq d(x,z)+d(z,y') < A+\delta.$$ 

Now applying the thin triangle condition at $x$, we have $d(x',[x,y']) \leq \delta$, hence if $w$ is the projection of $x'$ on $[x,y']$, then
$$ d(x,y') = d(x,w) + d(w,y') \geq d(x,x') + d(x',y') - 2\delta = d(x',y') + A -2\delta. $$
Combining these two inequalities gives $A+\delta > d(x',y') + A -2\delta$, hence $d(x',y') < 3\delta$.
Since $y$ was an arbitrary element of the ball $B$, we conclude that the projection of $B$ on $\alpha$ has diameter at most $6\delta$. 

\medskip
\noindent  $(1) \implies (5)$.  Suppose $\alpha$ is $D$-contracting. Let $x=\alpha(t)$ with $t > r > D$.  
Set $y_1=\alpha(t-r)$ and $y_2=\alpha(t+r)$.  Suppose $\beta$ is a path from $y_1$ to $y_2$ lying outside the ball $B(x,r)$.  For any point $w \in \beta$, let   $\beta_{w}$ denote the sub-path of $\beta$ from $w$ to $y_2$. 

Note that the projection of $\beta$ on $\alpha$ contains the interval $[y_1,y_2]$. In particular, there exists $z \in \beta$ that projects to $x$.  Let $z_1$ be a point in $\beta_z$ at distance $r$ from $z$ and let $x_1$ be its projection on $\alpha$.  Since $d(z,z_1) =r \leq d(z,x)$, 
the $D$-contracting hypothesis implies the $d(x,x_1) \leq D$ and hence  $d(x_1,\beta) \geq d(x,\beta)-D \geq r-D$.

Now repeat this process starting at $x_1$.  That is, choose $z_2 \in \beta_{z_1}$ at distance $r-D$ from $z_1$ and let $x_2$ be its projection on $\alpha$.   Since $d(z_1,z_2) =r-D \leq d(z_1,x_1)$,  the $D$-contracting hypothesis implies that $d(x_1,x_2) \leq D$ and $d(x_2, \beta) \geq r-2D$.
We can repeat this process $k=\lfloor \frac{r}{D} \rfloor < \frac{r}{D}-1$ times to get a sequence of points $z=z_0, z_1,z_2, \dots ,z_k$ on $\beta$ satisfying
$$ |\beta| \geq \sum_{i=0}^{k-1} d(z_i,z_{i+1})  = \sum_{i=0}^{k-1} (r-iD) = kr-\frac{k(k-1)}{2}D 
\geq \frac{r^2}{2D} -D. $$
We conclude that $ldiv_{\alpha}$ is at least quadratic.
 \end{proof}

 \section{The Contracting Boundary}\label{sec:contractingboundary}
 
 In this section we introduce the contracting boundary of a CAT(0) space $X$.  First, we recall the definition of the visual boundary and some basic properties.  For more details, see  \cite{bridsonhaefliger}, Section II.8.
\emph{We assume from now on that $X$ is complete.}

Two geodesic rays $\gamma,\gamma': [0, \infty) \ra X$ are said to be \emph{asymptotic} if there exists a constant $K$ such that $d(\gamma(t), \gamma'(t)) < K$ for all $t > 0,$ or equivalently, if they have bounded Hausdorff distance.  It is immediate that being asymptotic gives an equivalence relation on rays.  The \emph{visual boundary} of $X$, denoted $\partial X,$ is defined as the set of equivalence classes of geodesic rays.   The equivalence class of a geodesic ray $\gamma$ will be denoted $\gamma(\infty).$

It is an elementary fact that for $X$ a complete CAT(0) space and $x_0 \in X$ a fixed base point, every equivalence class can be represented by a unique geodesic ray emanating from $x_0$.  
One natural topology on  $\partial X$ is the \emph{cone topology}.  In this topology, a neighborhood basis for $\gamma(\infty)$ is given by all open sets of the form: 
$$
 U(\gamma,r,\epsilon) = \{ \alpha(\infty) \in \partial X \mid \alpha \mbox{ is a geodesic ray at $x_0$ and}\; \forall t<r, d(\alpha(t),\gamma(t)) <\epsilon \}
$$
In other words, two geodesic rays are close together in the cone topology if they have representatives starting at the same point which stay close (are at most $\epsilon$ apart) for a long time (at least $r$).  It is easy to verify that this topology is independent of choice of base point.  Moreover, if $X$ is  proper (i.e., closed balls in $X$ are compact), then $\partial X$ is compact.  
 
It follows from Lemma \ref{lem:constants} and Theorem \ref{theorem:equivalence} that if  $\alpha$ and $\beta$ are asymptotic geodesics, then $\alpha$ is contracting if and only if $\beta$ is contracting.  A more elementary proof of this fact can be found in \cite{bestvinafujiwara}, where the following is proved.

 \begin{lemma}\cite[Lemma 3.8] {bestvinafujiwara} \label{stability}  There exists a constant $D'$, depending only on $D$ and $C$, such that if $[a,b]$ and $[c,d]$ are geodesic segments with $d(a,c), d(b,d) \leq C$ and $[a,b]$ is $D$-contracting, then $[c,d]$ is $D'$-contracting.
\end{lemma} 

Hence, by Lemma \ref{stability}, we have a well defined notion of a point $\alpha(\infty)$ in the visual boundary being contracting or non-contracting. 
Define the \emph{contacting boundary} of a complete CAT(0) space $X$ to be the subset of the visual boundary 
consisting of 
$$ \partial_cX= \{\alpha(\infty) \in \partial X \mid \textrm{$\alpha$ is contracting} \}  $$
As before, we can fix a base point $x_0$ in $X$ and represent each point on $\partial_cX$ by a unique contracting ray based at $x_0$.  

\subsection{Topology on $\partial_cX$}
One possible topology on $\partial_cX$ is the subspace topology induced by the cone topology on $\partial X$.  For our purposes, however, a topology which takes account of the contracting constant is more natural, as well as more useful.
Fix a basepoint $x_0 \in X$.  For any natural number $n$, let
$\partial_c^nX_{x_{0}}$ denote the subspace of $\partial X$ consisting of points represented by some  $n$-contracting ray emanating from the fixed basepoint $x_{0}.$  That is,
 $$\partial_c^nX_{x_{0}}= \{[\gamma] \in \partial X | \gamma(0)=x_{0}, \gamma \mbox{ is an $n$--contracting ray } \}$$

Notice that there is an obvious inclusion map $i:\partial_{c}^{m}X_{x_{0}} \ra \partial_{c}^{n}X_{x_{0}}$ for all $m<n.$  Accordingly, we can consider the contracting boundary $\partial_{c}X$ as the direct limit, $\varinjlim_{n\in N} \partial_{c}^{n}X_{x_{0}}$, and equip the contracting boundary with the direct limit topology.   
 
It should be cautioned that the choice of basepoint $x_{0}$ effects  the contraction constant of our representative ray for a point in $\partial_{c} X$.  That is, for distinct base point $x_0, x_1 \in X$,  the subspaces $\partial_{c}^{n}X_{x_{0}}$ and $\partial_{c}^{n}X_{x_{1}}$ need not be the same.  Nonetheless, as we will see in Lemma \ref{lem:topology}  below, the direct limit topology on $\partial_cX$ is, in fact, independent of the choice of basepoint.  

 \begin{lemma} \label{lem:closed} For all $n,$ $\partial_c^nX_{x_0} \subset \partial X$ is closed with respect to the cone topology.
 \end{lemma}
 
 \begin{proof}
Let $\{\alpha_{i}\}$ be any sequence of $n$-contracting rays based at $x_0$  which converge to a ray $\beta.$  We 
need to show that $\beta$ is $n$-contracting.   For this, it suffices to verify that for any point $y$ not on $\beta$,  the projections
$z=\pi_\beta(y)$ and $y_i=\pi_{\alpha_i}(y)$ satisfy $d(z,y_i) \to 0$ as $i \to \infty$.
 By definition of convergence, the distance from any finite segment of $\beta$  to $\alpha_i$ approaches zero as $i \to \infty$.
 Thus, replacing $y_i$ by its projection $z_i = \pi_\beta(y_i)$, it suffices to show that $d(z,z_i) \to 0$ as $i \to \infty$.
 
Given any $\epsilon >0$, choose $i$ such that $\beta$ and $\alpha_i$ are $\epsilon$-close on a segment which includes
$z,z_i, y_i$.  Then the distances from $y$ to $\beta$ and from $y$ to $\alpha_i$ differ by at most $\epsilon$,  and hence  
$|d(y,z) - d(y, z_i) | < 2\epsilon$.  Consider the triangle $\Delta(y,z,z_i)$.  The Euclidean comparison triangle has angle $\geq \frac{\pi}{2}$ at $z$ (since $z$ is the projection of $y$ on $\beta$) so the edge lengths must satisfy
  $$ d(y,z) ^2 + d(z,z_i)^2 \leq d(y, z_i)^2 \leq (d(y,z) + 2\epsilon)^2.$$
Letting $\epsilon \to 0$, it follows that $d(z,z_i) \to 0$ as $i \to \infty$
\end{proof}

\begin{lemma}\label{lem:topology}
The direct limit topology on $\varinjlim_{n\in N} \partial_{c}^{n}X_{x_{0}}$ is independent of the choice of basepoint $x_{0}.$
\end{lemma}
\begin{proof}
Given two asymptotic rays $\gamma$ and $\gamma'$ emanating from $x_{0}$ and $x_{1},$ respectively, using CAT(0) convexity (property (C3) of \ref{lem:cat}) in conjunction with the fact that a bounded convex function is constant, it follows that $\gamma,\gamma'$ have Hausdorff distance at most $d(x_0,x_1)$.   In particular, by Lemma \ref{stability}, it follows that if $\gamma'$ is $c'$--contracting, then $\gamma$ is $c$--contracting with $c$ depending only on $c'$ and $d(x_{0},x_{1})$.
 In other words, the identity map gives an inclusion 
$$i: \partial_c^nX_{x_{1}} \ra\partial_c^{f(n)}X_{x_{0}}$$
 where $f: \N \ra \N$ is a non-decreasing function.  

To see that the direct limit topology on $\partial_{c} X$ is independent of the choice of base point,  first note that by Lemma \ref{lem:closed}, a subset of $\partial_c^nX_{x_{0}}$ is closed in $\partial_c^nX_{x_{0}}$ if and only if it is closed in $\partial X$, and likewise for $\partial_c^nX_{x_{1}}$.   Let $V$ be closed in $\varinjlim_{n\in N} \partial_{c}^{n}X_{x_{0}}.$  That is, $\forall n \; V^{n}_{x_{0}} := V \cap \partial_{c}^{n}X_{x_{0}}$ is closed in $\partial X.$   Applying Lemma \ref{lem:closed} again, we see  that
 $$V^{n}_{x_{1}} := V \cap \partial_{c}^{n}X_{x_{1}} = V^{f(n)}_{x_0} \cap \partial_{c}^{n}X_{x_1}$$ 
 is also closed in $\partial X$, so $V$ is closed in $\varinjlim_{n\in N} \partial_{c}^{n}X_{x_{1}}.$  
 
 By symmetry, the converse is also true.  That is, a set is closed in $\varinjlim_{n\in N} \partial_{c}^{n}X_{x_{0}}$ if and only if it is closed in $\varinjlim_{n\in N} \partial_{c}^{n}X_{x_{1}}$.
 \end{proof}

Hereafter, we will assume that the topology on $\partial_{c} X$ is the direct limit topology.  When convenient, we will also assume that the basepoint is fixed, omit it from the notation and write $\partial_{c} X = \varinjlim_{n\in N} \partial_{c}^{n}X.$

It is immediate that any set which is open (equivalently closed) in the subspace topology on $\partial_{c}X$ is also open (equivalently closed) in the direct limit topology $\partial_{c}X.$  On the other hand, as we will see in Example \ref{ex:strips} below,  the direct limit topology can be strictly finer than the subspace topology.

\subsection{Some examples}\label{sec:examples}

Before studying properties of contracting boundaries in general,  we consider some illuminating examples.  

First consider the case where
 $X=\R^{n}$ for $n\geq 2,$ or more generally, $X$ is the product of two unbounded CAT(0) spaces. 
 In this case, every geodesic is contained in a flat, hence $\partial_{c}X = \emptyset.$
At the other extreme is the case where
 $X$ is a CAT(0), $\delta$-hyperbolic space, so every geodesic ray is $D$-contracting where $D$ depends only on $\delta$.  In this case, $\partial_{c} X=\partial X$ and the direct limit topology is the same as the cone topology.   For example, $\partial_{c}\mathbb{H}^{2}=\mathbb S^{1}.$

Now let us combine these two extremes.  Let $X$ be the space obtained by gluing a half plane to a geodesic line $\beta$ in $\mathbb H^2$ and take the basepoint to be $\beta(0)$.  No ray lying in the half plane (including $\beta$ )is contracting, but all other rays lying in $\mathbb H^2$ are still contracting (though their contracting constants approach infinity as the rays approach $\beta$).  Thus, the contracting boundary $\partial_{c}X$ is an open interval.  This example illustrates the point    that unlike the visual boundary,  $\partial _{c}X$ need not be compact, even when $X$ is a proper space.


\begin{example}[Croke-Kleiner space] \label{ex:CK}  In \cite{crokekleiner}, Croke and Kleiner showed that boundaries of CAT(0) spaces are not invariant under quasi-isometry.  Their example involved the Salvetti complex $S_\G$ of a certain right-angled Artin group (RAAG).  We briefly recall their construction.  Let $\AG$ denote the RAAG associated to the graph $\G$ in Figure \ref{fig:CK}.  That is, 
$$\AG =\left < a,b,c,d \mid [a,b]=[b,c]=[c,d]=1 \right >.$$
The Salvetti complex $S_\G$ is the standard $K(\pi,1)$-space for this group, consisting of three tori with the middle torus glued to the other two tori along orthogonal curves corresponding to the generators $b$ and $c$.  The universal cover of this space, which we will denote by $X_\G$, is a CAT(0) cube complex. 

\begin{figure}[htpb] 
\centering
\includegraphics[height=4 cm]{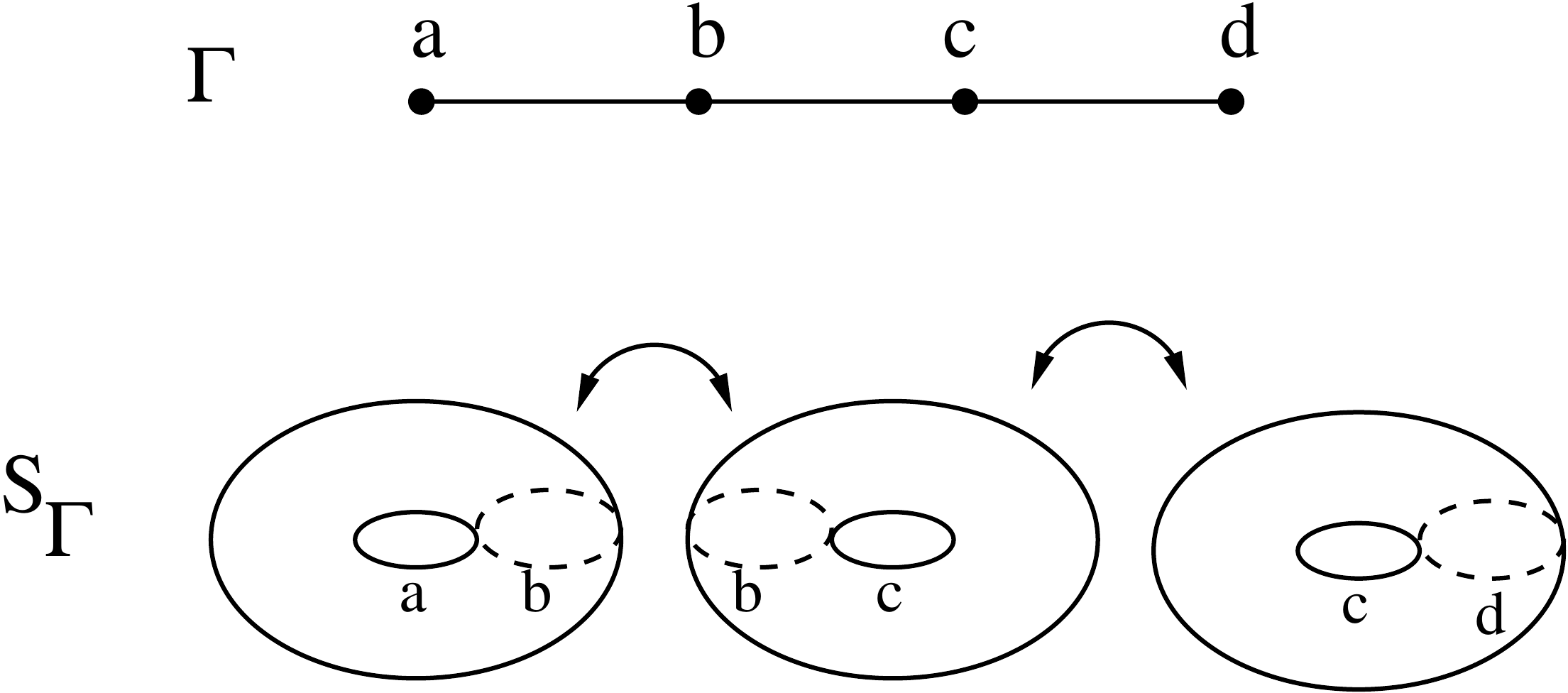}
\caption{Croke-Kleiner space} \label{fig:CK}
\end{figure}

Let $B_1$ denote the subspaces of $S_\G$ consisting if the union of the $(a,b)$-torus and the $(b,c)$-torus, and $\tilde B_1$ its inverse image in $X_\G$.  Then $B_1$ decomposes as a product hence each component of  $\tilde B_1$ has empty contracting boundary.  The same holds for $B_2$, the subspace of $S_\G$ consisting if the $(b,c)$-torus and the $(c,d)$-torus. Croke and Kleiner refer to these components as ``blocks".  Now consider a geodesic ray $\alpha$ in  $X_\G$.  Any segment of $\alpha$ contained in a block lies in a flat, hence if $\alpha$ is contracting, there must be a uniform bound on the length of such segments.  It will follow from the discussion of cube complexes in Section \ref{sec:cubecomplexes} below, that the converse is also true.  That is, $\alpha$ is contracting if and only if the length of segments of $\alpha$ lying in a single block is uniformly bounded.  
\end{example}

 \subsection{Basic properties of $\partial_cX$}

Now recall that if $X$ is a proper CAT(0) space, then the cone topology on the visual boundary $\partial X$, is compact.  The contracting boundary on the other hand, is not, in  general, compact. 
A space is said to be \emph{$\sigma$-compact} if it is a countable union of compact subspaces.  As an immediate consequence of Lemma \ref{lem:closed} we have the following:

 \begin{proposition} \label{prop:compact} If $X$ is proper, then  $\partial_cX$ is $\sigma$-compact.  
\end{proposition}

\begin{proof}
If $X$ is proper, then  $\partial X$ is compact, hence by Lemma \ref{lem:closed} so is $ \partial_{c}^{n}X$.  Since $\partial_{c}X=\bigcup_{n\in \N} \partial_{c}^{n}X,$ it is $\sigma$-compact.
\end{proof}

 Another useful property of the boundary of a hyperbolic space $X$ is the visibility property:  given any two distant points $\alpha(\infty), \beta(\infty)$ in $\partial X$, there is a geodesic line $\gamma$ in $X$ which is asymptotic to $\alpha$ at one end and asymptotic to $\beta$ at the other.  This is not the case for a CAT(0) space.  For example, if $X$ is the Euclidean plane, then the only time $\alpha(\infty), \beta(\infty)$ are visible from each other in this sense, is if they are antipodal on the boundary circle.  However, as we will see in the next proposition, points on the contracting boundary satisfy a strong visibility property.

\begin{proposition}\label{visibility}  Let $\alpha$ and $\beta$ be distinct rays based at $x_0$ and assume $\alpha$ is contracting. Then
\begin{enumerate}
\item the projection of $\beta$ on $\alpha$ is bounded.  
\item $\exists$ a bi-infinite geodesic $\gamma$ such that $\gamma (\infty)=\alpha (\infty)$ and $\gamma (-\infty)=\beta (\infty)$, 
\end{enumerate}
\end{proposition}

\begin{proof}  (1):  Suppose the projection of $\beta$ on $\alpha$ is unbounded.  Then there exists a sequence of real numbers $t_i \to \infty$ such that the projections $x_i$ of $\beta(t_i)$  leave every ball about $x_0$.  Applying the thin triangle condition (i) with $y=x_0$, $x=\beta(t_i)$, we see that there exists $\delta$ such that $[x_0,x_i]$ lies within $\delta$ of $\beta$ for all $i$.  But this contradicts our assumption that  $\alpha \neq \beta$.

(2): By part (1), the projection of $\beta$ on $\alpha$ is bounded hence lies within $B$ of $x_0$ for some constant $B$.   Take a sequence $t_i \to \infty$ and consider the geodesic segments $\gamma_i$ from $\beta(t_i)$ to $\alpha(t_i)$.  By the thin triangle condition (i), each of these segments  passes through the ball of radius $B+\delta$ about $x_0$.  It follows that the segments $\gamma_i$ stay uniformly bounded distance from $\alpha \cup \beta$, hence they converge to a bi-infinite geodesic as desired.
\end{proof}

\begin{corollary} $\partial_cX$ is a visibility space.  That is, any two points in $\partial_cX$ are connected by a (contracting) bi-infinite geodesic. 
\end{corollary}

The second statement of Proposition \ref{visibility} says that if $\alpha$ is contracting, then every point on the visual boundary of $X$ is ``visible" from $\alpha(\infty)$.  It is reasonable to ask whether this property characterizes contracting rays.  The answer is no, as the following example illustrates.

\begin{figure}[htpb] 
\centering
\includegraphics[height=4 cm]{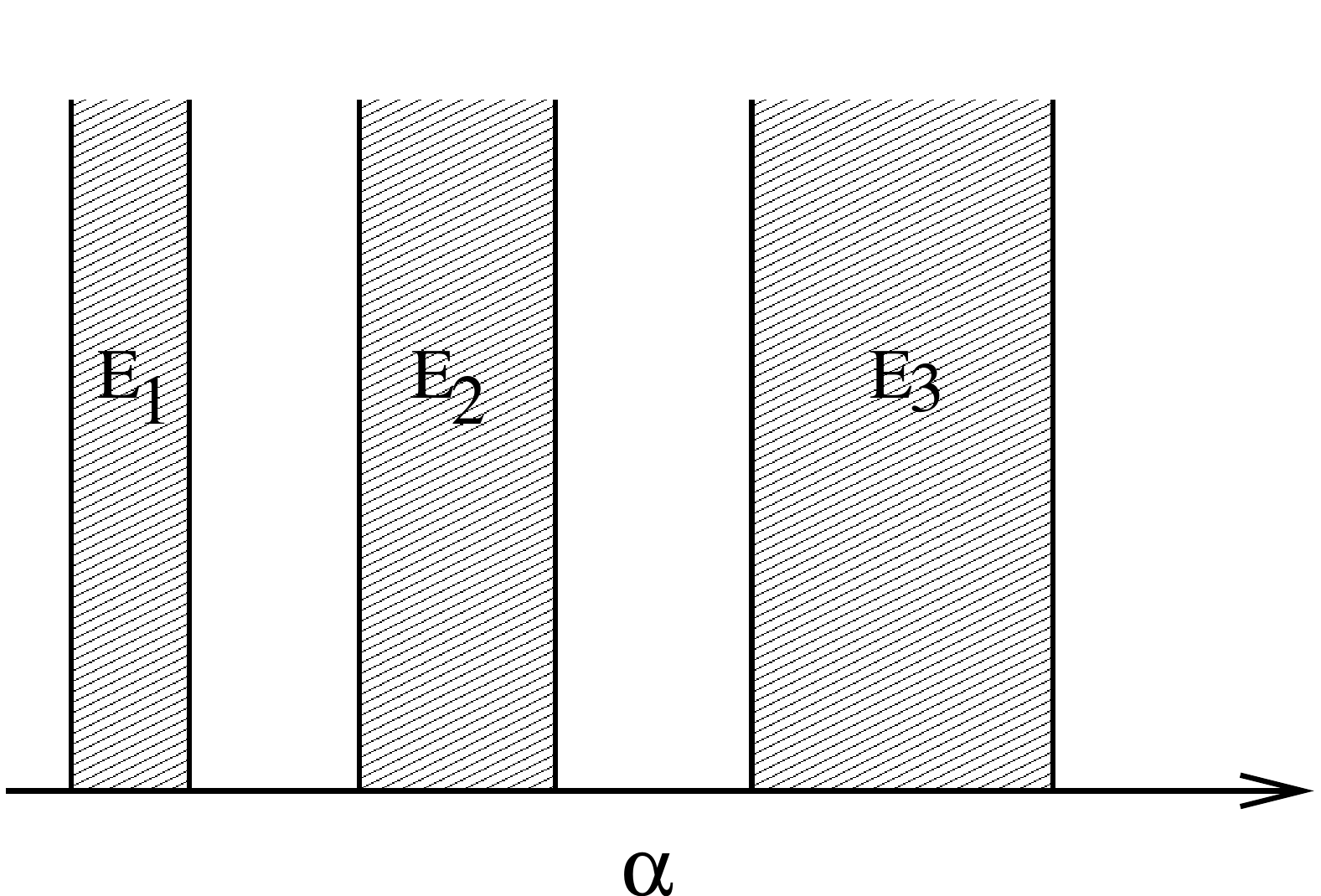}
\caption{} \label{fig:strips}
\end{figure}    

\begin{example}  
Let $X$ be constructed by starting with a ray $\alpha$, and attaching wider and wider Euclidean strips $E_n$ along $\alpha$, as in Figure \ref{fig:strips}.  Then $\partial X$ consists of the point $\alpha(\infty)$ together with one point $e_n(\infty)$ for each strip $E_n$.  It is easy to see that every point $e_n(\infty)$  is visible from $\alpha(\infty)$, but $\alpha$ is not contracting.
\end{example}

\subsection{Quasi-isometry invariance}\label{sec:qiinvariance}

In this section we prove our main theorem:  that quasi-isometries of CAT(0) spaces induce homeomorphisms of their contracting boundaries.  Let $f:X \rightarrow Y$ be a (K,L)--quasi-isometry.  Fix base points $x_0 \in X$ and $f(x_0) \in Y$. By Corollary \ref{cor:DtoD'}, $f$ induces a map 
$$\partial_{c}f: \partial_{c}X \rightarrow \partial_{c}Y$$  
which maps $\partial_{c}^{i}X$ into $\partial_{c}^{g(i)}(Y)$ for some non-decreasing function $g:\mathbb{N} \rightarrow \mathbb{N}$.

To prove that $\partial_cf$ is continuous, we will need the following a technical lemma regarding continuous maps in the direct limit topology.


\begin{lemma} \label{lem:technical} Let $X = \varinjlim X_{i}$ and $Y = \varinjlim Y_{i}$  with the direct limit topology.   Suppose $f:X \to Y$ is a function such that  $f(X_i) \subset Y_{g(i)},$ where $g:\mathbb{N} \rightarrow \mathbb{N}$ is a nondecreasing function. If $ f_{i} := f |_{X_{i}}$ is continuous for all $i$, then $f$ is continuous. 
\end{lemma}

\begin{proof}
Let $U\subset Y$ be open.  Then  $U_{g(i)} = U \cap Y_{g(i)}$ is open in $Y_{g(i)}$ for all  $i$.  Since $f_{i}$ is continuous, it follows that $f^{-1}(U) \cap X_{i} = f^{-1}(U_{g(i)}) \cap X_{i}=f_{i}^{-1}(U_{g(i)}) $  is open in $X_{i}.$ By definition of the direct limit topology,   $f^{-1}(U)$ is open in $X$.  
\end{proof}

\begin{theorem}\label{qi-invariance}  Let $f : X \to Y$ be a quasi-isometry of complete CAT(0) spaces.  Then 
$\partial_{c}f: \partial_{c}X \rightarrow \partial_{c}Y$ is a homeomorphism. 
\end{theorem} 

\begin{proof}  Let $g$ be a quasi-inverse of $f$ and let $\gamma \in \partial_cX$.  By definition, $\partial_cf(\gamma)$ lies bounded distance from the quasi-geodesic $f \circ \gamma$, hence  $(\partial_{c}f \circ \partial_{c} g)(\gamma)$ lies bounded distance from $(g \circ f) (\gamma)$.  Since $\gamma$ also lies bounded distance from $(g \circ f )(\gamma)$, we conclude that 
$\partial_{c}f \circ \partial_{c} g= id$, and similarly,  $\partial_{c}g \circ \partial_{c} f = id$.  
 It remains to prove continuity.
 
By Lemma \ref{lem:technical}, it suffices to prove that $\partial_{c}f_{i}: \partial_{c}^{i}X \rightarrow \partial_{c}^{g(i)}(Y)$ is continuous for all $i$. 
Let $\gamma \in \partial_{c}^{i}X$ and let $U$ be an open set in $\partial_{c}^{g(i)} Y$ of the form $U=U(\partial_{c}f(\gamma),r,\epsilon) \cap \partial_{c}^{g(i)}Y$.  We must show that the inverse image of $U$ under $\partial_{c}f_i$ contains an open neighborhood of $\gamma$  in  $\partial_{c}^{i}X.$
 
Say $f$ is a $(K,L)$-quasi-isometry.  Let $M=M(K,L)$ be the Morse constant with respect to $(K,L)$-quasi-geodesics with endpoints on a $g(i)$--contracting geodesic.  This is possible by Theorem \ref{theorem:equivalence}.  Let $V=V(\gamma, r',\epsilon) \cap \partial ^{i}_{c}X$.  We claim that for sufficiently large $r'$, this open set   satisfies $\partial_{c}f_{i}(V) \subset U$.

Let $\beta \in V.$  Then $f \circ \beta$,  and similarly $f \circ \gamma$, are $g(i)$--contracting, $(K,L)$-quasi-geodesics  in $Y$.  Moreover, by definition of $V,$ $d(\beta(r'),\gamma(r')) < \epsilon$,  so it follows that 
$$d((f \circ \beta)(r'),(f \circ \gamma)(r')) < K\epsilon +L.$$  
Moreover, by choosing $r'$ sufficiently large, we may assume $(f \circ \beta)(r')$ and $(f \circ \gamma)(r')$ are arbitrarily far from the basepoint $f(x_0)$, say distance at least $r'' >> r$.
Straightening $f \circ \beta,f \circ \gamma$ to geodesic rays $\widehat{\beta}:=\partial_{c}f(\beta)$ and $\widehat{\gamma}:=\partial_{c}f (\gamma),$  we then have 
$$d(\widehat{\beta}(r''), \widehat{\gamma}(r''))  < K\epsilon +L +2M.$$
For $r'' > \frac{r}{\epsilon}(  K\epsilon +L +2M)$, the CAT(0) convexity property then guarantees that 
$d(\widehat{\beta}(t), \widehat{\gamma}(t))  < \epsilon$ for all $t < r$.  This proves the claim.
 \end{proof}

In particular,  Theorem \ref{qi-invariance} allows us to define the contracting boundary of a CAT(0) group $G$ as the homoemorphism type of the contracting boundary of any complete CAT(0) space on which $G$ acts properly, cocompactly, by isometries.

\begin {corollary}\label{for:groups}  If $G$ is a CAT(0) group, then $\partial_{c}G$ is  well-defined up to $G$-equivariant homeomorphism. 
\end{corollary}

We close this section with some remarks on the direct limit topology.  The reader may wonder why we chose to use the direct limit topology on $\partial_c X$ rather than the cone topology induced from $\partial X$.  We claim that the direct limit topology is both more useful and more natural.  Indeed, we do not know if the map $\partial_cf$ in Theorem \ref{qi-invariance} is continuous with respect to the cone topology.  

Moreover, as the following example demonstrates, the direct limit topology holds more information about the underlying space.  This example describes two CAT(0) spaces, whose visual boundaries are identical and whose contracting boundaries are both set-wise equal to their visual boundaries (i.e., every ray is contracting). Thus, with respect to the cone topology, they would have the same contracting boundaries. Yet, the contracting boundaries of the two spaces equipped with the direct limit topology are distinct, reflecting the fact that the two spaces are not quasi-isometric.  

\begin{example}\label{ex:strips}  Let $T$ be the tree formed by a single horizontal line $L$ with a vertical  line $V_n$ attached at each integer point $n$ on $L$.  We can identify $T$ with the subspace of $\R^2$,
$$T = \{ (x,0)  \mid x \in \R\} \cup \{ (x,y) \mid x \in \Z, y \in \R\}. $$
Let $X$ be the space obtained by gluing a Euclidean strip of width $|n|$ along the line $V_n$.  So $X$ can be viewed as the subspace of $\R^3$,
$$X=  \{ (x,0,0)  \mid x \in \R\} \cup \{ (x,y,z) \mid x \in \Z, y \in \R, 0 \leq z \leq |x| \}. $$
Every geodesic ray to infinity in $X$ lies in the tree $T$ so the visual boundaries of $X$ and $T$ are homeomorphic.  Moreover, every such ray is contracting in both $X$ and $T$, so set-theoretically, the contracting boundaries also agree. 
On the other hand, consider the rays $\alpha_n$ formed by traveling along $[0,n] \subset L$ and then along $[0, \infty) \subset V_n$.  In $\partial_c T=\partial T$, these rays converge to $[0,\infty) \in L$.  In $\partial_c X$, however, these rays form a closed set since only finitely many lie in  $\partial_c^n X$ for any fixed $n$.  Indeed, the topology on  $\partial_c X$ is the discrete topology, hence it is not homeomorphic to $\partial_c T$. 

This example also illustrates the fact that the direct limit topology can be strictly finer that the cone topology since the set
$\{\alpha_n\}$ is closed in the direct limit topology on  $\partial_cX,$ but not the cone topology.
\end{example}

\section{Cube Complexes}
\label{sec:cubecomplexes}

In this section we discuss properties of contracting boundaries of CAT(0) cube complexes.  Let $X$ be a CAT(0) cube complex.  We recall the definition of a hyperplane in $X$ and refer the reader to \cite{haglundwise} or \cite{sageev} for additional details.  Define an equivalence relation on the set of midplanes of cubes generated by the condition that two midplanes are equivalent if they share a face.  A  \emph{hyperplane} is defined to be the union of all the midplanes in an equivalence class.  Hyperplanes in a CAT(0) cube complex are geodesic subspaces and divide the space into two components.

We will assume throughout this section that $X^{(1)}$,  the one-skeleton of $X$, has bounded valence $\nu$, or equivalently, that the ball of radius one about any vertex intersects at most $\nu$ hyperplanes.  It follows, more generally, that the number of hyperplanes intersecting any ball of radius $r$ is bounded by a function $\nu(r)$.  We will call such a cube complex \emph{uniformly locally finite}.  Note that this assumption implies that $X$ is both locally finite and finite dimensional.

\subsection{Criteria for contracting rays}
Two hyperplanes $H_1,H_2$ are said to be \emph{strongly separated} if they are disjoint and no hyperplane intersects both $H_1$ and $H_2$.  This notion was first introduced by Behrstock and the first author in \cite{behrstockcharney}, and it is featured  in the rank rigidity theorem of Caprace-Sageev \cite{capracesageev}.  In \cite{behrstockcharney},  it is shown that a periodic geodesic in $X$ which crosses an infinite sequence of strongly separated hyperplanes has quadratic divergence.  For periodic geodesics, lower divergence is equivalent to divergence, so in conjunction with Theorem \ref{theorem:equivalenceexpanded}, it follows that a periodic geodesic that crosses an infinite sequence of strongly separated hyperplanes is contracting.  The converse, on the other hand, is not true.   For example,  a contracting geodesic can be contained in a hyperplane $H$ so that any two hyperplanes that cross the geodesic, also intersect $H$. 

To establish necessary and sufficient conditions for a geodesic in a CAT(0) cube complex to be contracting, including non-periodic geodesics, we will need a more general notion of separation.

\begin{definition}  Two hyperplanes $H_1,H_2$ are \emph{ $k$-separated} if they are disjoint and at most $k$ hyperplanes intersect both $H_1$ and $H_2$. 
 In particular,  $H_1,H_2$ are $0$-separated if and only if they are strongly separated.
\end{definition}

\begin{theorem} \label{cube case} Let $X$ be a uniformly locally finite CAT(0) cube complex.  There exist $r>0, k \geq 0$ (depending only on $D$ and $\nu$), such that a geodesic ray $\alpha$ in $X$  is $D$-contracting if and only if  $\alpha$ crosses an infinite sequence of hyperplanes $H_1,H_2,H_3, \dots$ at points $x_i= \alpha \cap H_i$ satisfying 
\begin{enumerate}
\item $H_i,H_{i+1}$ are $k$-separated and 
\item  $d(x_i,x_{i+1}) < r.$
\end{enumerate}
\end{theorem}

To prove this, we will need several lemmas.  We first consider the case where $\alpha$ is a $D$-contracting ray.

\begin{lemma}\label{hyp projection}  There exists a constant $R$ (depending only on $D$ and $\nu$) such that if $\alpha$ is a $D$-contracting ray and $x_1,x_2$ are two points on $\alpha$ at distance at least $R$, then the segment $[x_1,x_2]$ of $\alpha$ crosses a hyperplane $H$ whose projection on $\alpha$ lies entirely within $[x_1,x_2]$.
\end{lemma}

\begin{proof}  Suppose $H$ crosses $[x_1,x_2]$ at a point $y$.  If some point $z \in H$ projects to $x_1$, then the thin triangle condition implies that the geodesic from $z$ to $y$ passes through the ball $B(x_1, \delta)$.  There are a bounded number of such hyperplanes, say at most $K$, with $K$ depending only on $\nu$.  Likewise, at most $K$ hyperplanes  $H$ that cross $[x_1,x_2]$ have projection containing $x_2$.  Set $R=2(K+1) \nu.$  Since $d(x_1, x_2) > R$, it follows that $[x_1,x_2]$ crosses more than $2K$ hyperplanes, and in particular it crosses a hyperplane whose projection contains neither $x_1$ nor $x_2$.  
\end{proof}

\begin{lemma}\label{wide projection1} Let $\alpha$ be a $D$--contracting geodesic segment and $x,y \in X$ two points not on $\alpha$.  Set $a=d(x, \pi_\alpha(x)), b=d(y, \pi_\alpha(y))$,  and  $c=d(\pi_\alpha(x),  \pi_\alpha(y)).$  If $c \geq 2D$, then 
$$ a+b+c-2D  \leq  d(x,y) \leq a+b+c. $$
\end{lemma}

\begin{proof}  
Let $x'$ be the point on $[x,y]$ at distance $a$ from $x$ and let $y'$ be the point on $[x,y]$ at distance $b$ from $y$.  Then the projections of $[x,x']$ and $[y,y']$ on $\alpha$ have length at most $D$, hence
$d(x',y') \geq d(\pi_\alpha(x'),  \pi_\alpha(y')) \geq c-2D.$  Thus 
$d(x,y)= a+b+d(x',y') \geq a+b+c-2D$.  The other inequality follows from the triangle inequality.  
\end{proof}

\begin{lemma}\label{wide projection}
Let $\alpha$ be a $D$--contracting geodesic segment.  If $\beta=[x,y]$ is a geodesic segment whose projection on $\alpha$ has length at least $4D,$ then  $\pi_{\alpha}(\beta) \subset N_{5D}{\beta}$.  
\end{lemma}

\begin{proof}
Let $a,b,c$ and $x',y'$ be as in the previous lemma, so  $d(x,y) = a+b+d(x',y') \leq a+b+c$, hence $d(x'y') \leq c$.  Now set
$a'=d(x', \pi_\alpha(x')), b'=d(y', \pi_\alpha(y'))$,  and $c'=d(\pi_\alpha(x'),  \pi_\alpha(y')).$  Applying Lemma \ref{wide projection1} to $x',y'$ gives 
$$ c \geq d(x',y') \geq a'+b'+c' -2D \geq a'+b'+c-4D.$$
It follows $a'+b' \leq 4D$.  By convexity of the metric the maximum distance between $[x',y']$ and its projection
$[\pi_\alpha(x'),\pi_\alpha(y')]$  is attained at one end, so the Hausdorff distance between them is at most $4D$. 
The lemma follows since every point on $[\pi_\alpha(x),  \pi_\alpha(y)]$ lies within $D$ of a point on 
$[\pi_\alpha(x'),  \pi_\alpha(y')]$.
\end{proof}

\begin{lemma}\label{k-sep}  Let $\alpha$ be a $D$--contracting geodesic ray.   Then there exists $k$ such that if $\alpha$ crosses two hyperplanes $H_1,H_2$ whose projections are distance at least $4D$,  then $H_1,H_2$ are $k$-separated.
\end{lemma}

\begin{proof}  First note that $H_1,H_2$ are disjoint since their projections on $\alpha$ are disjoint.  Suppose a hyperplane $H$ crosses both $H_1$ and $H_2$.  Let $y_i \in H \cap H_i$.  Then the geodesic $\beta$ from $y_1$ to $y_2$ lies in $H$.  By Lemma \ref{wide projection}, $\beta$ passes through the ball of radius $5D$ about any point on $\alpha$ between the two projections.  The number of such hyperplanes is bounded. 
\end{proof}

For the converse implication, we will need the following lemma which is an analogue of Lemma 2.3 from \cite{behrstockcharney}.  

\begin{lemma} \label{distance}  There exists constants $c_1,c_2 > 0$, depending only on $r,k$ and $\nu$ satisfying the following.   Suppose $H_1,H_2$ are $k$-separated hyperplanes and $x_1, y_1 \in H_1$, $x_2, y_2 \in H_2$ are points with $d(x_1,x_2) \leq r$, then
$$d(y_1,y_2) \geq c_1(d(x_1,y_1)+d(x_2,y_2)) - c_2$$
\end{lemma}

\begin{proof}    This is easy to see if we use the $d^{(1)}$-metric instead of the CAT(0) metric, where $d^{(1)}(a,b)$ is  the number of hyperplanes separating $a$ and $b$.  Since the CAT(0) metric is bounded above and below  by linear functions of the $d^{(1)}$-metric, depending only on $\nu$, the result will follow.

 Consider the path from $y_1$ to $y_2$ made up of the 3 geodesics segments $[y_1,x_1] \cup[x_1,x_2] \cup  [x_1,y_2]$.  Any hyperplane crosses this path at most 3 times and a hyperplane separates $y_1$ and $y_2$ if and only if it crosses the path an odd number of times.  Throwing out hyperplanes that cross both $[x_1,y_1]$ and $[x_2,y_2]$ (there are at most $k$ such)  and hyperplanes that cross $[x_1,x_2]$ and one of the other sides (there are at most $\nu r$), we are left with at least 
 $d^{(1)}(x_1, y_1)+d^{(1)}(x_2, y_2) - 2k - \nu r$ hyperplanes which cross this path exactly once.  
 
 Hence $d^{(1)}(y_1, y_2) \geq  d^{(1)}(x_1, y_1)+d^{(1)}(x_2, y_2) - 2k  - \nu r$ and the lemma follows.
 \end{proof}

\begin{proof}[Proof of Theorem \ref{cube case}]   First assume $\alpha$ is $D$-contracting and let $R$ be as in Lemma \ref{hyp projection}  Divide 
$\alpha$ into a sequence of segments $\alpha_1, \beta_1, \alpha_2, \beta_2, \dots $ where $\alpha_i$ has length $R$ and $\beta_i$ has length $4D$.  By Lemma \ref{hyp projection}, each $\alpha_i$ intersects a hyperplane $H_i$ whose projection lies entirely in $\alpha_i$.  It then follows from Lemma \ref{k-sep} that there exists $k$ such that $H_i, H_{i+1}$ are $k$-separated for all $i$, and by construction, their intersections with $\alpha$ are at distance at most $4D+2R$. 

Conversely, suppose $\alpha$ crosses an infinite sequence of hyperplanes $H_1,H_2, \dots$ satisfying conditions  (1) and (2) of the theorem.  We will show that the lower divergence of $\alpha$ is quadratic.  Let $z_0 = \alpha (t)$, and suppose $s<t $.  Consider a path $\beta$ from $z_1= \alpha(t-s)$ to $z_2= \alpha(t+s)$  which stays outside the ball of radius $s$ about $z_0$.   Every hyperplane crossing $[z_1,z_2]$ must also cross $\beta$.  Since the $H_i$ do not intersect, $\beta$ crosses these hyperplanes in the same order.  

Say $H_i,H_{i+1}, \dots H_{i+n}$ cross $\alpha$ between $\alpha(t-\frac{s}{2})$ and 
$\alpha(t+\frac{s}{2})$.  By assumption (2), $n \geq \frac{s}{r}$.  For $i \leq j \leq i+n$, set $x_j = H_j \cap \alpha$ 
and $y_j = H_j \cap \beta$.  Then
$$ d(x_j,y_j) \geq d(y_j, z_0) - d(z_0,x_j) \geq s-\frac{s}{2} = \frac{s}{2}. $$
It now follows from Lemma \ref{distance} that $d(y_j,y_{j+1})$ is bounded below by a linear function of $s$.  Since $\beta$ crosses at least $\frac{s}{r}$ such hyperplanes, the length of $\beta$ is bounded below by a quadratic function of $s$.   
\end{proof}

\begin{example}[Croke-Kleiner revisited] We return to the Croke-Kleiner space $X_\G$ discussed in Example \ref{ex:CK}.  Recall that $X_\G$ is the universal cover of the Salvetti complex shown in Figure \ref{fig:CK}.  We can identify the 1-skeleton of $X_\G$ with the Cayley graph of the RAAG $\AG$ and represent geodesic rays by edge paths (which cross the same sequence of hyperplanes as the CAT(0) geodesic ray).  The edges dual to any hyperplane are all labeled by the same generator and two hyperplanes which cross each other must be labelled by commuting generators.  Two hyperplanes contained in the same block are never $k$-separated for any $k$, since blocks are products.  Hence to be contracting, an edge path $\alpha$ must spend bounded amount of time in any single block.  Conversely, any segment of $\alpha$ not contained in a block must contain both an edge labelled $a$ and an edge labelled $d$.  The hyperplanes dual to these two edges are strongly separated since no generator commutes with both $a$ and $d$.  We conclude that $\alpha$ is contracting if and only if it spends a bounded amount of time in each block, or equivalently, if and only if it corresponds to an infinite word $w=w_0\,a\,w_1\,d\,w_2\,a\,w_3\,d \dots$ where the lengths of the $w_i$ are uniformly bounded.

The interest in this space stems from the fact that Croke and Kleiner \cite{crokekleiner} showed that modifying the metric on $X_\G$ by skewing the angle between the $b$ and $c$ curves, so that the $(b,c)$-cubes become parallelograms, changes the homeomorphism type of the boundary.  More generally, J.~Wilson \cite{wilson}  showed that any two distinct angles between the $b$ and $c$ curves gave rise to non-homeomorphic boundaries.  More recently, Y.~Qing \cite{qingthesis} showed that leaving the angles orthogonal but changing the side lengths of the cubes can also affect the homeomorphism type of the boundary.  More precisely,  she showed that the identity map, which is a quasi-isometry between these two metrics, does not induce a homeomorphism on the boundary.

In  \cite{crokekleiner} and \cite{wilson}, the change in the topology of the boundary that occurs when angles are skewed can be seen  in the way in which the boundaries of the blocks intersect.   In  \cite{qingthesis}, the change occurs in the components of the boundary corresponding to rays which spend longer and longer time in successive blocks.  None of these points appear in the contracting boundary.  Indeed, this example suggests that the restriction to contracting rays is optimal if one seeks a quasi-isometry invariant structure in the boundary.
\end{example}


\section{Applications to Right Angled Coxeter Groups}\label{application}

Since by Theorem  \ref{qi-invariance}, contracting boundaries are quasi-isometry invariants, they can be used to distinguish between quasi-isometry classes of groups.  In this section we will use contracting boundaries  to show that certain right-angled Coxeter groups are not quasi-isometric.  Some quasi-isometry classes of right-angled Coxeter groups are easily distinguished using number of ends, hyperbolicity, relative hyperbolicity, or divergence.  We will describe an example of two  groups that cannot be distinguished by any of these criteria, but have non-homeomorphic contracting boundaries. 

%


We begin with some preliminaries before constructing our main example.  Recall that for $\Gamma$ a finite simplicial graph with vertex set $V=\{v_{i}\}$ and edge set $E=\{(v_{i},v_{j})\},$ the \emph{right angled Coxeter group (RACG) associated to $\Gamma$}, denoted $W_{\Gamma},$ is the group with presentation
\begin{equation}\label{eq:presentation}W_{\Gamma}= \left<v_{i}\in V \mid  v_{i}^{2}=1, [v_{i},v_{j}]=1 \iff (v_{i},v_{j}) \in E\right>.
\end{equation}

Associated to $W_\G$ is a CAT(0) cube complex, the Davis complex for $\Sigma_\G$, on which $W_\G$ acts properly and cocompactly. It is constructed as follows.  Since every generator of $W_\G$  has order two, the Cayley graph has two, oppositely directed edges labelled $v_i$ connecting the vertices  $w$ and $wv_i$.  Collapsing these edges to a single, unoriented edge, the resulting graph is the 1-skeleton of $\Sigma_\G$.  Now attach cubes wherever possible, that is, fill in an $n$-cube wherever the graph contains the 1-skeleton of the $n$-cube.  The resulting cube complex is the Davis complex for $W_\G$.  Note that every hyperplane in $\Sigma_{\Gamma}$ intersects edges with a unique label $v,$ and hence it makes sense to define the \emph{type of a hyperplane} in $\Sigma_{\Gamma}$ to be this unique label.

\begin{figure}[htpb] 
\centering
\includegraphics[height=5 cm]{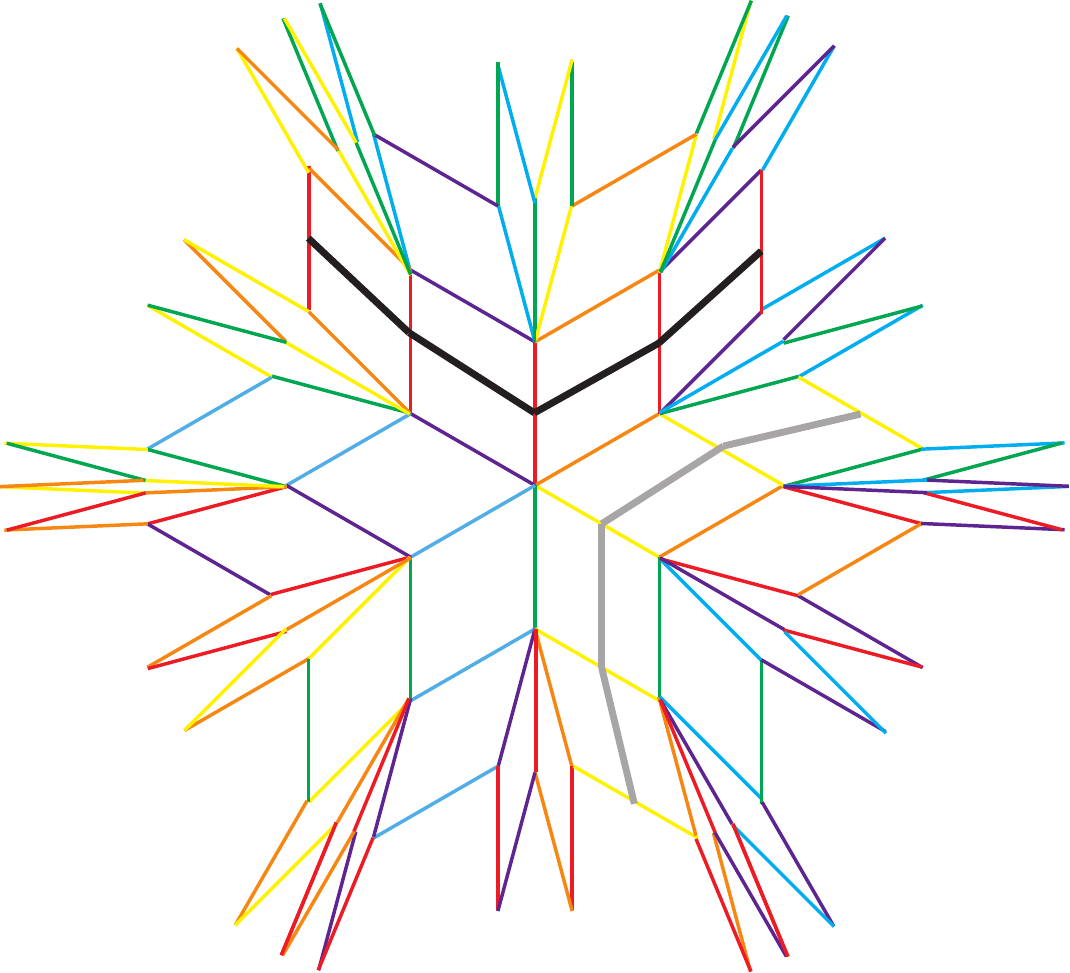}
\caption{A local picture of $\Sigma_{\Gamma}.$  The black and grey hyperplane are 1-separated.}\label{fig:hexagon}
\end{figure}    

First consider two easy examples.  
Let the graph $\Gamma$ be a regular hexagon, and $ W_{\Gamma}$ its corresponding right angled Coxeter group.  $W_\G$ acts as a reflection group on the hyperbolic plane  $\mathbb{H}^{2}$ with fundamental domain a right-angled hexagon. 
Hence the Davis complex $\Sigma_{\Gamma}$ is quasi-isometric to $\mathbb{H}^{2}$.  This can be seen directly by noting that the 2-complex dual to the tiling of  $\mathbb{H}^{2}$ by right-angled hexagons is (combinatorially) identical to the Davis Complex.  See Figure \ref{fig:hexagon}.   It follows that $\partial_{c}(W_{\Gamma})\cong \partial_{c}(\mathbb{H}^{2})=\mathbb{S}^1.$  In particular, notice that hyperplanes of different types corresponding to nonadjacent vertices are either 0-separated  or 1-separated, see for instance the hyperplanes in Figure \ref{fig:hexagon}. Similarly, notice that distinct hyperplanes of the same type are also at most 1-separated.

Next, let $\Omega$ be the graph $\Gamma$ with the three long diagonals of the hexagon added in as edges.  The resulting graph  $\Omega$ is the complete bipartite graph $K_{3,3}.$  In particular, $\Omega$ is a join of two subgraphs $G_{1},G_{2}$ where each $G_{i}$ is a discrete graph with three vertices.  It follows that the corresponding right angled Coxeter group is a direct product, $W_{\Omega}=W_{G_{1}} \times W_{G_{2}},$ and hence $\partial_{c}(W_{\Omega})=\emptyset.$  Note that since $\partial_{c}(W_{\Gamma}) \cong \mathbb{S}^1$ while $\partial_{c}(W_{\Omega})=\emptyset,$ it follows by Theorem \ref{qi-invariance} that $W_{\Gamma}$ and $W_{\Omega}$ are not quasi-isometric, but this was already clear since the former is   $\delta$-hyperbolic, while the latter is not.

\begin{center}
\begin{figure}
\tikzstyle{dot}=[circle,draw,fill, inner sep=0,minimum size=1.5mm]

\begin{tikzpicture}

\node[dot] (g1bl)  at (0,0) {};
\node[dot] (g1ml) [above of=g1bl] {}
	edge (g1bl);
\node[dot] (g1tl) [above of=g1ml] {}
	edge (g1ml);
\node[dot] (g1br) [right of=g1bl] {}
	edge (g1bl) ;
\node[dot] (g1mr) [above of=g1br] {}
	edge (g1br);
\node[dot] (g1tr) [above of= g1mr] {}
	edge (g1mr)
	edge (g1tl);

\node[dot] (g2bl)  at (2,0) {};
\node[dot] (g2ml) [above of=g2bl] {}
	edge (g2bl);
\node[dot] (g2tl) [above of=g2ml] {}
	edge (g2ml);
\node[dot] (g2br) [right of=g2bl] {}
	edge (g2bl) 
	edge (g2tl);
\node[dot] (g2mr) [above of=g2br] {}
	edge (g2br)
	edge (g2ml);
\node[dot] (g2tr) [above of= g2mr] {}
	edge (g2mr)
	edge (g2tl)
	edge (g2bl);
\draw [draw=gray, inner sep=1pt, thick] (g2mr) circle (2.5mm);
\draw [draw=gray, inner sep=1pt, thick]  (g2tl) circle (2.5mm);
\draw [draw=gray, inner sep=1pt, thick]  (g2bl) circle (2.5mm);
\draw [draw, inner sep=1pt, thick]  (g2ml) circle (2.5mm);
\draw [draw, inner sep=1pt, thick]  (g2tr) circle (2.5mm);
\draw [draw, inner sep=1pt, thick]  (g2br) circle (2.5mm);

\node[dot] (c1) [label=120:$c_1$] at (4.5,2) {};
\node[dot] (c2) [below of=c1, label=left:$c_2$] {}
	edge (c1);
\node[dot] (c3) [below of=c2, label=220:$c_3$] {}
	edge (c2);
\node[dot] (c4) [right of=c1, label=above:$c_4$] {}
	edge (c1);
\node[dot] (c5) [below of=c4, label=left:$c_5$] {}
	edge (c4);
\node[dot] (c6) [below of= c5, label=below:$c_6$] {}
	edge (c5)
	edge (c3);
\node[dot] (d1) [right of=c4, label=45:$d_1$] {}
	edge (c4)
	edge (c6);
\node[dot] (d2) [below of=d1, label=right:$d_2$] {}
	edge (d1)
	edge (c5);
\node[dot] (d3) [below of=d2, label=320:$d_3$] {}
	edge (d2)
	edge (c4)
	edge (c6);
	
\node[dot] (a1) [label=120:$a_1$] at (8.25,2) {};
\node[dot] (a2) [below of=a1, label=left:$a_2$] {}
	edge (a1);
\node[dot] (a3) [below of=a2, label=220:$a_3$] {}
	edge (a2);
\node[dot] (a4) [right of=a1, label=above:$a_4$] {}
	edge (a1);
\node[dot] (a5) [below of=a4, label=left:$a_5$] {}
	edge (a4);
\node[dot] (a6) [below of= a5, label=below:$a_6$] {}
	edge (a5)
	edge (a3);
\node[dot] (b1) [right of=a4, label=above:$b_1$] {}
	edge (a4);
\node[dot] (b2) [below of=b1, label=left :$b_2$] {}
	edge (b1);
\node[dot] (b3) [below of=b2, label=below:$b_3$] {}
	edge (b2)
	edge (a6);
\node[dot] (b4) [right of=b1, label=45:$b_4$] {}
	edge (b1)
	edge (b3);
\node[dot] (b5) [below of=b4, label=right:$b_5$] {}
	edge (b4)
	edge (b2);
\node[dot] (b6) [below of=b5, label=320:$b_6$] {}
	edge (b5)
	edge (b1)
	edge (b3);

\draw (.5, -.75) node  {$\Gamma$};
\draw (2.5, -.75) node {$\Omega=\textcolor{gray}{G_1} \ast G_2$};
\draw (5.54, -.75) node {$\Gamma_1$};
\draw (9.8, -.75) node {$\Gamma_2$};

\end{tikzpicture}
\caption{Graphs of $\Gamma, \Omega, \Gamma_1, \Gamma_2$.}
\label{fig:graphs}
\end{figure}
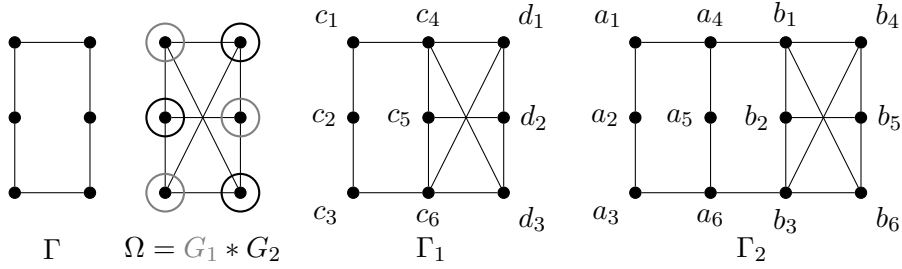
\end{center}


Our main example is constructed by amalgamating copies of $W_\Gamma$ and $W_\Omega$.  Let $\Gamma_{1}$ be the graph obtained by connecting a copy of $\Gamma$ and a copy of $\Omega$ as in Figure \ref{fig:graphs}.  Similarly let $\Gamma_{2}$ be the graph obtained by connecting two copies of $\Gamma$ and a copy of $\Omega$ as shown in in Figure \ref{fig:graphs}.  We will show that $W_{\Gamma_{1}}$ and $W_{\Gamma_{2}}$ are not quasi-isometric by proving that $\partial_{c}(W_{\Gamma_{1}})$ is totally disconnected, whereas $\partial_{c}(W_{\Gamma_{2}})$ contains a circle.  

\subsection{$\partial_{c}W_{\Gamma_{1}}$ is totally disconnected.}
Let $C$ denote the subgroup of $W_{\Gamma_1}$ generated by $\{c_4,c_5,c_6\}$.  Then  $W_{\Gamma_1}$ is the amalgamated product $W_{\Gamma_1}= W_{\G} \ast_{C} W_{\Omega} $, where 
 $\Gamma$ is subgraph spanned by the $c_i$-vertices and $\Omega$ is the subgraph spanned by $\{c_{4},c_{5},c_{6},d_{1},d_{2},d_{3}\}$.   

As noted above,  on its own, $W_\G$ is quasi-isometric to a hyperbolic plane and has contracting boundary a circle.
Viewed as a subgroup of $W_{\G_1}$, however, the picture changes.   The group $W_\Omega$ is a product hence 
every geodesic in $W_\Omega$ bounds a flat.  In particular, viewed as a subgroup of $W_{\G_1}$, an infinite word in $W_\G $ with arbitrarily long segments in $C = W_\G \cap W_\Omega$ is not contracting.  In fact, using Theorem \ref{cube case} it can be seen that these  are precisely the non-contracting rays in $W_\G$.  Moreover, for a contracting ray, the length of the maximal subword in $C$ determines the contraction constant. 
 
Note that the set of geodesics in $W_\G$ which have subwords in $C$ of arbitrarily long length, and hence are not contracting, is dense in  $\partial W_\G =\mathbb{S}^1.$  This follows since any geodesic word $w$ can be approximated by a sequence of geodesics which are identical to $w$ for an arbitrarily long amount of time, then remain in $C$ from there on.  It follows that the subspace of  $\partial_{c} W_{\G_1}$  formed by rays in $W_\G$ is a circle with a dense set of points removed.  In particular, it is totally disconnected.  
 
We now consider the general case of a geodesic ray in $W_{\G_1}$.  Let $T$ be the Bass-Serre tree for the amalgamated product decomposition, so $T$ is a bipartite graph with vertices labelled by cosets of  $W_\G$ and $W_\Omega$.  
 Let $x_0$ be the base point in  the Davis complex $\Sigma_{\G_1}$ corresponding to the identity vertex in the Cayley graph and let $v_0$ be the vertex of $T$ labelled $W_\G$.   For a geodesic segment or ray $\alpha$ in $\Sigma_{\G_1}$, based at $x_0$, $\alpha$ determines a path $I_\alpha$ in $T$, based at $v_0$ which we call the \emph{itinerary} of $\alpha$.  Note that paths which are sufficiently close have the same itinerary, so the itinerary of a point in $\partial_c \Sigma_{\G_1}$ is well-defined.  Note also that a contracting geodesic either has infinite itinerary or finite itinerary ending in a coset of $W_\G$.

 For two paths $I_1,I_2$ in $T$ based at $v_0$, write $I_1 \leq I_2$ if $I_1$ is an initial subpath of $I_2$.  Set
\begin{eqnarray*}
 U(I) &=& \{ \alpha \in \partial_c\Sigma_{\G_1} \mid I \leq I_\alpha \} \\
 \widehat  U(I) &=& \{ \alpha \in \partial_c\Sigma_{\G_1} \mid I = I_\alpha \} 
 \end{eqnarray*}
 Observe that for any finite path $I$,  $U(I)$ is both open and closed  since paths with sufficiently close initial segments have the same initial itinerary.  
  
Suppose $\alpha, \beta \in \partial_c\Sigma_{\G_1}$ have distinct itineraries, $I_\alpha \neq I_\beta$.  Then there is a finite path $I$ that is an initial segment of one of the itineraries, say $I_\alpha$, but not the other, i.e., $\alpha \in U(I)$ while $\beta \notin U(I)$.  Since $U(I)$ is open and closed,  $\alpha$ and $\beta$ do not lie in the same connected component.  We conclude that connected components must lie entirely in $\widehat U(J)$ for some (possibly infinite) path $J$.

Suppose $J$ has finite length and assume $J$ terminates in a coset of $W_\G$ (otherwise  $\widehat U(J)$ is empty).  Then $\widehat U(J)$ is a translate of the contracting boundary of $W_\G$ discussed above, namely a circle with a dense set removed.  Thus, it is totally disconnected.  

If $J$ is infinite,  we claim that  $\widehat U(J)$ consists of a single point.  For suppose $u, w$ are infinite, contracting words with itinerary $J$.  Since $T$ is bipartite, $J$ contains vertices of type $\Omega$ at arbitrarily large distances from the basepoint.   Thus $u$ and $w$ contain initial subwords lying in a coset $gW_\Omega$ for arbitrarily long $g$. Say $ga$ is a subword of $u$ and $gb$ is a subword of $w$, with $a,b \in W_\Omega$.  The length of $a^{-1}b$ is
bounded by the contracting constants of $u$ and $w$ since every hyperplane crossed by $a^{-1}b$ must be crossed by either $u$ or $w$.  It follows that the distance between $ga$ and $gb$ is uniformly bounded for all such $g$.  Thus,  $u$ and $w$ represent the same point at infinity.

We conclude that connected components of  $\partial_c\Sigma_{\G_1}$ are singletons.

 \subsection{$\partial_{c}W_{\Gamma_{2}}$ contains a circle.}
In order to see that $\partial_{c}W_{\Gamma_{2}}$ contains a circle, it suffices to show that there exists $D$ such that any geodesic ray in the Cayley graph of $W_{\Gamma_{2}}$ which lies in the subgroup $W_\G$ generated by the set $\{a_{i}\}$  is $D$-contracting.   For then, the circle corresponding to $\partial W_\G$ is a subset of  $\partial_{c}W_{\Gamma_{1}},$ and since the contraction constant is uniform, the usual topology on this circle subset remains intact.  However,  this follows from an application of Theorem \ref{cube case}, since hyperplanes which are at most $1$-separated in $\Sigma_ {\G}$ remain at most $1$-separated in $\Sigma_{\G_2}$.

\vskip .5in

\bibliographystyle{amsplain}
\bibliography{bib}
\end{document}